\documentclass[10pt,leqno]{amsart}
\usepackage[utf8]{inputenc}
\usepackage{indentfirst,csquotes}

\usepackage{amsmath}
\usepackage[english]{babel}
\usepackage{dirtytalk}
\usepackage{fancyhdr} 
\usepackage{amsthm}
\usepackage{nicematrix}
\usepackage{hyperref}
\usepackage{amssymb}
\usepackage{mathtools}
\usepackage{amsmath}
\usepackage{latexsym}
\usepackage{esint}
\usepackage{soul}
\usepackage{wasysym}
\usepackage{pst-eucl}
\usepackage{comment}
\usepackage{hyperref}
\usepackage{mathrsfs}
\usepackage{blkarray}
\hypersetup{
    colorlinks=true,
    linkcolor=blue,
    citecolor=teal
    }
\let\div\undefined\DeclareMathOperator{\div}{div}
\newcommand{\ps}[2]{\left\langle #1,#2\right\rangle}

\newcommand{\norm}[1]{\left\lVert#1\right\rVert}
\newcommand{\nor}[1]{\left\lvert#1\right\rvert}

\newtheorem{Theorem}{Theorem}
\numberwithin{Theorem}{section}
\newtheorem{Proposition}[Theorem]{Proposition}
\newtheorem{Lemma}[Theorem]{Lemma}
\newtheorem{Corollary}[Theorem]{Corollary}
\theoremstyle{definition}

\theoremstyle{remark}
\newtheorem{Remark}[Theorem]{Remark}
\numberwithin{equation}{section}
\makeatletter
\newenvironment{customproof}[1]{%
  \par\pushQED{\qed}%
  \normalfont\topsep6pt \trivlist
  \item[\hskip\labelsep\itshape
    Proof of #1\@addpunct{.}]\ignorespaces
}{%
  \popQED\endtrivlist\@endpefalse
}
\makeatother

 \DeclareMathOperator{\R}{\mathbb{R}}
 \DeclareMathOperator{\N}{\mathcal{N}}

\newcommand{\p}{\partial}

\newcommand{\obal}{\ensuremath{\mathbb{R}^n} \setminus B_r}

\newcommand{\frlap}{\ensuremath{(-\Delta)^s}}
\newcommand{\eee}{\ensuremath{\varepsilon}}

\newcommand{\FF}{\mathscr{F}}

\begin{document}
\title[Fractional Alt-Caffarelli-Friedman monotonicity formula]{On a fractional Alt-Caffarelli-Friedman-type monotonicity  formula} 
\author[]{Fausto Ferrari,  Davide Giovagnoli and Enzo Maria Merlino}

\address{Fausto Ferrari: Dipartimento di Matematica, Universit\`a di Bologna, Piazza di Porta \\ S.Donato 5, 40126, Bologna-Italy}
\email{fausto.ferrari@unibo.it}
\address{Davide Giovagnoli: Dipartimento di Matematica, Universit\`a di Bologna, Piazza di Porta \\ S.Donato 5, 40126, Bologna-Italy}
\email{d.giovagnoli@unibo.it}
\address{Enzo Maria Merlino: Dipartimento di Matematica, Universit\`a di Bologna, Piazza di Porta \\ S.Donato 5, 40126, Bologna-Italy}
\email{enzomaria.merlino2@unibo.it}
\date{\today}
\subjclass[2020]{35R35, 35R11, 34K37.}
\keywords{ACF monotonicity formula, fractional Laplacian, $s$-harmonic functions, fractional gradient, Bochner-type formulas.}

\maketitle

{\bf Abstract.} In this paper, by exploiting mean value properties of $s$-harmonic functions, we introduce some monotonicity formulas in the nonlocal setting. 
Our approach is purely nonlocal and does not rely on the extension technique. 
As a byproduct we also establish interior nonlocal gradient estimates and a nonlocal analogue of the Bochner identity.

\thispagestyle{fancy}
\fancyhead{} 
\fancyfoot{}

\fancyfoot[L]

\section{Introduction}

In this work we study the fractional counterpart of the local functional 
\begin{equation} \label{ACF-intro}
    J_{ACF}(u,R):=\frac{1}{R^{2}} \int_{B_{R}} \frac{|\nabla u|^{2}}{|x|^{n-2}} d x.
\end{equation}
In the local case, this functional enjoys monotonicity properties and can be applied in the regularity theory of PDEs to prove gradient estimates, see \cite[Section 2.2]{PSU_book}. This functional, together with its generalizations associated with different classes of operators, has been extensively studied, see for instance \cite{ferrari2020new,FG,garofalo2023note,PSU_book}. The main conceptual novelty of this paper is that the monotonicity is established entirely within the nonlocal framework through mean value properties, without relying on any extension argument.

Our motivation to study \eqref{ACF-intro} comes from its connection with the Alt-Caffarelli-Friedman (ACF) monotonicity formula that  plays an important role in free boundary problems. 
Originally introduced by Alt, Caffarelli, and Friedman in their seminal paper \cite{ACF} the formula is applied to prove the regularity result for solutions to two-phase Bernoulli-type elliptic free boundary problems. 
More precisely, the classical ACF monotonicity formula says that for every two nonnegative continuous subharmonic functions $u_{1}, u_{2} \in H^{1}(B_1)$, such that $u_{1}\cdot u_{2}=0$ and $u_{1}(0)=u_{2}(0)=0$, the functional 
\begin{equation}\label{ACF-intro2}
I_{ACF}(u_1,u_2,R):=\left(\frac{1}{R^{2}} \int_{B_{R}} \frac{|\nabla u_{1}|^{2}}{|x|^{n-2}} d x\right)\left(\frac{1}{R^{2}} \int_{B_{R}} \frac{|\nabla u_{2}|^{2}}{|x|^{n-2}} d x\right),
\end{equation}
is bounded for $0<R<1$ and is an increasing function of $R\in(0,1)$, where, as usual, $B_R$ denotes the Euclidean ball of radius $R$ centered at $0$. 
Several generalizations of this formula can be found in the literature, see, for instance, \cite{caffarellimonotonicity,caffarelli1988harnack,caffarelli2002some,caffarelli1998gradient,caffarelli2005geometric,conti2005asymptotic,edquist2008parabolic,matevosyan2011almost,soave2022anisotropic,teixeira2011monotonicity,velichkov2014note}.

In the literature, we may find several results on free boundary problems ruled by the fractional Laplacian where, in principle, a nonlocal ACF formula might be useful. In light of some weak connections with the present paper, let us recall \cite{de2012regularity}, where the authors study a local one-phase Bernoulli problem associated to the Caffarelli–Silvestre extension of the fractional Laplacian and, along the same line of research, the results on the extended two-phase version available in \cite{All12,AG24}. In addition, it is worth pointing out the following papers \cite{ros2025improvement,ros2026optimal,ros2026regularity}, where a purely nonlocal one-phase Bernoulli problem is addressed using new techniques that bypass the use of monotonicity formulas. For the fractional obstacle problem, in its extended formulation, we also refer to \cite{allen2015two}, where a sharp analysis concerning the nonexistence of two phases is carefully discussed.

Due to the nonlocal nature of the fractional Laplacian, it is currently unclear whether a companion functional of \eqref{ACF-intro2} exists for the two-phase case. Hence, as a first step, it is natural to begin by considering a counterpart of \eqref{ACF-intro} which  represents a valuable tool for obtaining regularity estimates for nonlocal operators, in this regard, see Theorem \ref{Thm:IntGradientEst}.

In the following, we outline our main idea.
 In the local case, the Poisson kernel for the Laplace operator in the ball $B_r$ is given by  $$ K_r(x,y):=\frac{1}{n \omega_n r}\frac{r^2-|x|^2}{|x-y|^n},$$ for $x\in B_r$ and $y\in\partial B_r$, where $\omega_n$ is the $n$-Lebesgue measure of the unit ball in $\R^n$, see for instance \cite{GT}. Thus, given a sufficiently regular function $u$ in the ball $B_1$, applying the coarea formula yields
\begin{align*}
\int_{B_R}\frac{| \nabla u |^2}{|x|^{n-2}}\,dx &=\int_0^R \int_{\partial B_r}\frac{|\nabla u|^2}{r^{n-2}}\,d\mathcal{H}^{n-1}\,dr\\ 
&= n \omega_n \int_0^R r \fint_{\partial B_r}|\nabla u|^2\,d\mathcal{H}^{n-1}\,dr=n \omega_n \int_0^R r \int_{\partial B_r}K_{r}(0,y)|\nabla u(y)|^2\,d\mathcal{H}^{n-1}\,dr,
\end{align*}
for $0<R<1$. 
Hence, the functional in \eqref{ACF-intro} can be written as
\begin{equation}\label{ACF-local-poi}
    J_{ACF}(u,R)= \frac{1}{R^2}\int_{B_R} \frac{|\nabla u|^2 }{\nor{y}^{n-2}} \, dy =\frac {n \omega_n }{R^2}\int_0^R r \int_{\partial B_r}K_{r}(0,y)|\nabla u(y)|^2\,d\mathcal{H}^{n-1}\,dr.
\end{equation} 
This remark motivates the introduction of the following functional
\begin{equation}\label{ACF-nonlocal-general}
     J_{ACF}^{s}(u,R) := \frac{n \omega_n}{R^{1+s}} \int_{0}^{R} r^s \int_{B_r^c}  K_{r}^s (0,y)  g_{u}(y)  dy \, dr,
\end{equation}
as the nonlocal counterpart of \eqref{ACF-intro},
where $K_r^s$ is the Poisson kernel for the fractional Laplacian, defined as
\begin{equation}
	 K^s_{r}(0,y) :=  a_{n,s} \Bigg (\frac {r^2}{|y|^2-r^2}\Bigg)^s \frac {1}{|y|^n}, \label{poissondefn}
	\end{equation}
see \cite[Chapter I, Section 6, n. 23]{Landkof}, or Subsection \ref{subsection_mean} in this paper, for the definition of $ K^s_{r}$. Moreover, $g_u:\R^n\to\R$ represents the nonlocal version of the squared norm of the gradient in the local functional \eqref{ACF-local-poi}. We will discuss this point later on, in the Introduction.

Our approach is inspired by an idea developed in \cite{FF_Bumi}, in the local setting, where the authors relate the monotone increasing behavior of $J_{ACF}$ with the (sub)mean value property of (sub)harmonic functions. More precisely, in the local case, if $u$ is harmonic, then $|\nabla u|^2$ is subharmonic and the monotonicity of $J_{ACF}(u,R)$ can be deduced from the monotonicity of $\fint_{\partial B_R}|\nabla u|^2$, for sufficiently small $R$.

Similarly, in the nonlocal setting, we prove the monotonicity of \eqref{ACF-nonlocal-general} whenever $g_u$ is a $s$-subharmonic function.  
Although this condition on $g_u$ arises naturally, it is not always straightforward to verify it in the nonlocal framework. It remains unclear which quantity best replaces $|\nabla u|^2$ in the nonlocal context.

As a consequence, the function $g_u$ has a central role in this approach. In the nonlocal setting, however, there are several possible alternatives to consider. In this paper, we focus on two examples that appear to be reasonable candidates, although both of them might seem partially unsatisfactory from different perspectives.

The first candidate is closely related to the variational formulation of the fractional Laplacian. Specifically, we set $g_u=\frac{1}{2}G_{u}$  where $G_u$ is defined as
\begin{equation*}\label{def-G-into-1}
G_{u}(y):= C_{n,s} \int_{\R^n} \frac{(u(y)-u(\eta))^2}{\nor{y-\eta}^{n+2s}}  \, d\eta \quad \text{for } y\in\R^n,
\end{equation*}
where $C_{n,s}$ is a suitable normalization constant controlling the behavior of $G_u$ as $s\to 1^{-}$. The function $G_{u}$ is a reasonable counterpart of the squared norm of the gradient in the variational formulation of the fractional Laplacian, since
\begin{equation} \label{eq:energiaWs2}
    \frac 1 2 \int_{\mathbb{R}^n}G_{u}(y)dy=[u]^2_{W^{s,2}(\mathbb{R}^n)},
\end{equation}
where $[u]^2_{W^{s,2}(\mathbb{R}^n)}$ is the Gagliardo seminorm, defined in \eqref{def_seminorma_Gagliardo}.
Moreover, $G_u$ is useful in obtaining interior gradient estimates via integration by parts, since it naturally arises when computing $(-\Delta)^s u^2$, see Theorem \ref{Thm:IntGradientEst}.

An alternative choice for $g_u$ is motivated by the contributions to fractional calculus established in \cite{S19} and \cite{CS22,CS19,CS23}. In fact, several authors focused on a nonlocal notion of gradient known as the {\it distributional Riesz fractional gradient}. Originally introduced in the work of Riesz \cite{Marcel_Riesz} (see also the description of the fractional-order partial derivative in \cite{Samko_others}), the distributional Riesz fractional gradient provides a counterpart to the local gradient. More precisely, one considers
$$\nabla^{s} u(y)
:=\mu_{n, s} \int_{\R^{n}} \frac{(y - \eta)  (u(y) - u(\eta)) }{|y - \eta|^{n + s + 1}} \, d\eta,
\quad \text{for } y\in\R^n,$$
where $\mu_{n,s}$ is a suitable normalization constant controlling the behavior of $\nabla^{s}$ as $s \to 1^{-}$. For further details, see \cite{S19} as well as \cite[Section 1]{CS19}. Following this idea, we consider the choice $g_u(y)=|\nabla^s u(y)|^2$,  where
\begin{equation}\label{def-G-intro-2}
 |\nabla^s u(y)|^2=\mu_{n, s}^2\sum_{i=1}^n\left(\int_{\mathbb{R}^n}\frac{(y_i-\eta_i)(u(y)-u(\eta))}{|y-\eta|^{n+s+1}}d\eta\right)^2,\quad
y\in\R^n.
\end{equation}
Furthermore, in analogy with \eqref{eq:energiaWs2}, whenever $\nabla^s u \in L^{2}(\R^n,\R^n)$, there exists a constant $C$, depending only on $n$ and $s$, such that
\[
\|\nabla^s u(y)\|_{L^{2}(\R^n,\R^n)}^2 = C[u]^2_{W^{s,2}(\R^n)},
\]
see, for instance, \cite[Proposition 2.8]{alicandro2025topological}.

We recall that many regularity results for fractional operators established through monotonicity formulas are obtained by considering local degenerate problems derived by lifting the operator via the Caffarelli-Silvestre extension; see, for instance, \cite{Terracini_Tortone_Vita,terracini2013uniform,terraciniVZ2016uniform,terracini2018asymptotic}. An intrinsic approach that does not rely on a lifting procedure, particularly when such an extension is not available, may be useful for gaining a better understanding of the properties of the solutions. To the best of our knowledge, a nonlocal counterpart of \eqref{ACF-intro} is not yet known.

The main result of the paper is the following.

\begin{Theorem}[Monotonicity formula] \label{Thm:sACF}
 Let $s \in (0,1)$ and $\varepsilon,\delta >0$ be fixed. Assume that $u \in C_{loc}^{s+\varepsilon}(\R^n) \cap L_{s}^2(\R^n)$ and that $g_u$ is either $g_u=\frac{1}{2}G_u$ or $g_u=|\nabla^s u|^2$. In the case $g_u = |\nabla^s u|^2$, we additionally assume that $u \in L_{s/2}^1(\R^n)$. Suppose that
\[ 
 g_u \in C_{loc}^{2s + \delta}(\R^{n}) \cap L_s^1(\R^n) \;
 \text{ and } \; (-\Delta )^s g_{u} \leq 0  \text{ in a neighborhood of the origin}.
 \] 
 Then the map 
 \begin{equation*}
     R\mapsto J_{ACF}^{s}(u,R):= \frac{n \omega_n }{R^{1+s}}\int_{0}^{R} r^s \int_{B_r^c} K^s_r(0,y)g_u(y)\, dy \, dr
\end{equation*}
is monotone increasing for $R$ sufficiently small.
\end{Theorem}

Furthermore, in Theorem \ref{Thm:ACF-lim} and Theorem \ref{Thm:ACF-lim-cs} we prove that for $R>0$ it holds that
$$J_{ACF}^s(u,R)\to J_{ACF}(u,R), \quad  \text{as } s \rightarrow 1^{-},$$
whenever $g_u$ is either $g_u=\frac{1}{2}G_u$ or $g_u=|\nabla^s u|^2$.

\medskip
\textbf{Applications of the main result.}
In the following we describe the main applications of our results.
As already mentioned, the definition of $G_u$ also allows one to obtain interior \textit{nonlocal gradient} estimates.
\begin{Theorem}[Interior \textit{nonlocal gradient} estimates] \label{Thm:IntGradientEst}
Let $s\in (0,1)$. Let $u\in C^{2}_{loc}(\R^n)\cap L^2_s(\R^n)$, be such that 
    \begin{equation*}
         \begin{cases}
        (-\Delta)^s u = f \quad &\textit{in } B_{1} \\
        u = 0  \quad &\textit{on }  B_{1}^c,
    \end{cases}
    \end{equation*}
    with $f \in L^\infty(B_1)$. Assume that $G_u\in C^{2s+\delta}_{loc}(\R^n)\cap L^1_s(\R^n)$ for some small $\delta>0$, and $(-\Delta)^s G_u\leq 0$ in $B_{1}$.  Then 
    \begin{equation}\label{claim2-intro}
    \|G_u\|_{L^\infty(B_{1/2})}^{1/2}\leq C_0\left(\|u\|_{L^\infty(B_{1})}+\|f\|_{L^\infty(B_{1})}\right),
    \end{equation}    
    for some constant $C_0=C_0(n,s)>0.$
\end{Theorem}

Concerning the assumption on the $s$-subharmonicity of $g_u$, we point out that in the local case, as discussed in \cite{garofalo2023note}, the Bochner identity represents a natural link between the harmonicity of a function and the subharmonicity of the squared norm of its gradient
 \begin{equation} \label{BochnerOriginal}
      (-\Delta)(|\nabla u|^2) =  2\ps{\nabla u }{\nabla((-\Delta) u)} - 2 \norm{D^2 u }^2 .
 \end{equation}
 
 An immediate consequence of \eqref{BochnerOriginal} is that if $\Delta u = c$ in some open set $\Omega$, with $c$ constant, then $|\nabla u|^2$ is subharmonic in $\Omega$. In the nonlocal setting, the same property holds if $\Omega=\R^n$, that is $(-\Delta)^s u =c$ in $\R^n$ with both choices of $g_u$ considered, see Remarks \ref{rmk:s-sub_Gu} and \ref{rmk:s-sub_nabla_s}.
With this aim, in Section \ref{section_Bochner}, we investigate a nonlocal Bochner-type identity involving our choice of $g_u$.
In particular, under suitable hypotheses on $u$ we obtain a nonlocal analogue of \eqref{BochnerOriginal}, see Proposition \ref{lem:cc1} and Proposition \ref{prop:bochner-cs} respectively.

The compatibility condition $(-\Delta)^s g_u \le 0$ seems generally difficult to verify explicitly, however as a consequence of the aforementioned results we provide a family of examples 
in which the condition is satisfied, see Corollary \ref{cor:semilinear} and
Corollary \ref{coroll-ACF-csv2}.

 As a main application, in the case $g_u=\frac{1}{2}G_u$, we use that $G_u$ is globally $s$-subharmonic for the class of semilinear equations $(-\Delta)^s u = f(u)$ driven by a monotone nonincreasing nonlinearity to show a 
 Liouville-type rigidity result, see Theorem \ref{thm:liouville_unified}, which also provides an alternative proof for the classical $s$-harmonic case.

\medskip
\textbf{Plan of the paper.} 
The work is organized as follows. In Section \ref{preliminaries} we introduce the main notation and we recall the basic results we need to apply in the sequel.
In Section \ref{sec:main_result} we prove the main results when $g_u=\frac{1}{2}G_u$, pointing out that the scheme of the proofs applies when $g_u=|\nabla^su|^2$. Indeed, the companion proofs are collected in Section \ref{alternativedef}, where some properties of $\nabla^s$ are recalled as well.
Moreover, Section \ref{section_Bochner} is devoted to studying a nonlocal Bochner-type formula associated to both choices of $g_u$ and applications of our main result.
Finally, Appendix \ref{appendix} collects an auxiliary technical result useful for the convergence analysis.

\section{Notations and preliminary results}\label{preliminaries}
\subsection{Some nonlocal quantities}
Let $\Omega\subseteq\R^n$ be an open set and $s\in(0,1)$ with $n\geq 2$. Then we define the fractional Sobolev space $W^{s,2}(\Omega)$ as
\begin{equation*}
W^{s,2}(\Omega):=\left\{u\in L^2(\Omega)\Big|[u]_{W^{s,2}(\Omega)}<\infty\right\},
\end{equation*}
where 
\begin{equation} \label{def_seminorma_Gagliardo}
[u]_{W^{s,2}(\Omega)}:=\left( \frac{C_{n,s}}{2} \int_\Omega\int_\Omega\frac{|u(x)-u(y)|^2}{|x-y|^{n+2s}}\,dx\,dy\right)^\frac{1}{2}
\end{equation}
is the Gagliardo seminorm of $u$.

For $s\in(0,1)$, the $s$-fractional Laplacian of $u$ is the nonlocal operator defined as
\begin{equation}\label{def_laplace}
\begin{split}
		 \frlap u(x)\; :=\; & C_{n,s} \,P.V. \int_{\R^n} \frac{ u(x)- u(y)}{|x-y|^{n+2s}}  \, dy \\
					\;=\; & C_{n,s}\,  \lim_{\varepsilon \to 0} \int_{ B_{\varepsilon}^c(x)} \frac{u(x) - u(y)}{|x - y|^{n+2s}}\, dy,
		\end{split}
\end{equation}
where $u:\R^n \to \R$ and the constant $C_{n,s}$ is a dimensional constant that depends on $n$ and $s$, precisely given by 
\begin{equation}\label{def_c}
C_{n,s}= \left(\int_{\R^n} \frac{1-\cos(\zeta_1)}{|\zeta|^{n+2s}}\,d\zeta \right)^{\!-1}
\end{equation}
where $\zeta_1$ is, up to rotations, the first coordinate of $\zeta\in \R^n$, see \cite{di2012hitchhiker's}.
 In general, we will omit $P.V.$ in the notation. 
 For any $1\leq p<\infty$, we introduce the space $L^p_s$ as the weighted $L^p$ tail space defined by 
\[L^p_s (\R^n):=\Big\{ u \in L^p_{loc}(\R^n)\; : \;  \|u\|_{L^p_s (\R^n)}:= \left( \int_{\R^n}\frac{ |u(x)|^p}{1+|x|^{n+2s}} \, dx \right)^{1/p}<\infty\Big\}.\]
 For $s\in (0,1)$, the definition \eqref{def_laplace} is well posed for sufficiently smooth functions belonging to the weighted space  $L_s^1(\R^n)$.
Indeed, for $\varepsilon>0$, sufficiently small and $u \in L_s^1(\R^n) \cap C_{loc}^{0,2s+\varepsilon}(\R^n)$ (or $C_{loc}^{1,2s+\varepsilon-1}(\R^n)$ for $s\geq 1/2$), the fractional Laplacian is well defined as in \eqref{def_laplace}, see \cite[Proposition 2.1.4]{silvestre2007regularity}. In the following, we always write $C^{2s+\varepsilon}$ to denote both $C^{0,2s+\varepsilon}$ for $s<1/2$ and $C^{1,2s+\varepsilon-1}$ for $s\geq 1/2$. We usually use this regularity assumption to guarantee its pointwise meaning. We recall that the fractional Laplacian can be viewed as a pseudo-differential operator of symbol $|\xi|^{2s}$, namely if $u\in\mathcal{S}(\R^n)$  it holds
\begin{equation*}
(-\Delta)^s u (x)\, = \, \FF^{-1}(|\xi|^{2s}(\FF u(x)))
\end{equation*}
for any $x,\xi\in\R^n$.
 We recall that a function $u\in C^{2s+\eee}(\Omega)\cap L^1_s(\R^n)$, is $s$-harmonic (respectively $s$-subharmonic or $s$-superharmonic) in $\Omega$ if
$$ (-\Delta)^su=0\,(\text{resp.}\leq 0 \,\text{ or }\geq 0), \qquad \text{ in } \Omega.$$

Let $s\in(0,1)$ we define
\begin{equation}\label{def-G}
    G_{u}(y) := C_{n,s} \int_{\R^n} \frac{(u(y)-u(\eta))^2}{\nor{y-\eta}^{n+2s}}  \, d\eta, 
\end{equation}
where $u:\R^n\to \R$ and $C_{n,s}$ is the constant defined in \eqref{def_c}. The term $G_u$ in the nonlocal setting plays the role of the gradient norm squared. We motivate this choice observing that
$$
\frac 1 2 \int_{\mathbb{R}^n}G_{u}(y)dy=[u]^2_{W^{s,2}(\mathbb{R}^n)}
$$
that is the energy associated to the variational formulation of the fractional Laplacian. 

 Nevertheless, as for the fractional Laplacian, for $s\in(0,1)$ in order to have the integral of the right-hand side of \eqref{def-G} well defined, we will assume that $u \in L_s^2(\R^n) \cap C_{loc}^{s+\eee}(\R^n)$.

Moreover, the term $G_u$ naturally arises in the expansion of the fractional Laplacian of the square of the function $u$. Namely, it plays the role of the gradient norm squared in the classical identity for $\Delta(u^2)$, as clarified in the next result.

\begin{Lemma} \label{lemma:tec}
Let $s \in (0,1)$. Then for $u \in L_s^2(\R^n) \cap C_{loc}^{2}(\R^n)$ we have
    \begin{equation*}
        (-\Delta )^s u^2(x) = 2u(x) (-\Delta )^s u(x) - G_{u}(x), \quad \text{for }x \in \R^n.
    \end{equation*}
\end{Lemma}
\begin{proof}
By expanding the square and rearranging terms, we obtain
\begin{align*}
     (-\Delta )^s u^2(x) &= C_{n,s}\int_{\R^n} \frac{u^2(x)-u^2(y)}{|x-y|^{n+2s}} \, dy 
     = C_{n,s} \int_{\R^n} \frac{(u(x)+u(y))(u(x)- u(y))}{|x-y|^{n+2s}} \, dy \\ 
     &= C_{n,s} \, u(x) \int_{\R^n} \frac{u(x)- u(y)}{|x-y|^{n+2s}} \, dy 
     + C_{n,s} \int_{\R^n} \frac{u(y)(u(x)- u(y))}{|x-y|^{n+2s}} \, dy \\
     &= 2u(x) (-\Delta )^s u(x) - C_{n,s}\int_{\R^n} \frac{(u(x)- u(y))^2}{|x-y|^{n+2s}} \, dy
     =  2u(x) (-\Delta )^s u(x) - G_{u}(x).
\end{align*}
\end{proof}

\begin{Remark} \label{rmk:reg-for-Gu}
Unlike the case of the fractional Laplacian, the integral in \eqref{def-G} does not need to consider the principal value sense. Indeed, let us suppose that $u\in L^2_s(\R^n) \cap C_{loc}^{s+\varepsilon}(\R^n)$ for some $\varepsilon>0$. Fix $x\in B_1$, by the local H\"older continuity, there exists a constant $C>0$ such that for any $y \in B_R(x)$ it holds
$$\left|u(x)-u(y)\right|\leq C\left|x-y\right|^{s+\varepsilon}.$$
Outside the ball $B_R(x)$, estimating the denominator with $|x-y| \ge C(1+|y|)$ for some $C$ depending on $R$ and $x$,  we can make use of $u\in L^2_s(\R^n)$. So, splitting the integral, passing to polar coordinates centered at $x$, we obtain 
\begin{eqnarray*}
&& \int_{\R^n} \frac{(u(x)-u(y))^2}{|x-y|^{n+2s}}\, dy \\ 
&&\leq C \left(\int_{B_R(x)} \frac{|x-y|^{2(s+\varepsilon)}}{\,|x-y|^{n+2s}}\, dy + \|u\|_{L^{\infty}(B_1)}^2 \int_{ B_R^c(x)} \frac{1}{\,|x-y|^{n+2s}}\, dy + \int_{B_R^c(x)} \frac{u(y)^2}{(1+|y|)^{n+2s}}\, dy \right) \\ 
&&\leq C \left( \int_{0}^{R}\! \rho^{2\varepsilon-1} \, d\rho \, + \int_{R}^{+\infty}\! \rho^{-1-2s} \, d\rho  + \|u\|^2_{L^2_s(\R^n)} \right)  < \, +\infty.
\end{eqnarray*}
\end{Remark}

\subsection{Fractional mean value properties.} \label{subsection_mean}
In this subsection, we recall some well-known results concerning potential theory of the fractional Laplacian, see \cite{Landkof} and \cite{bucur2016some}.

The fundamental solution for the fractional Laplacian with pole at the origin, denoted by $\Phi_s$, is given by
   	\begin{equation}\label{sol-fond}
		\Phi_s(x) :=
			 \displaystyle \kappa_{n,s}{|x|^{-n+2s}} 
	\end{equation}
 where $\kappa_{n,s}$ is a constant depending only on $n$ and $s$, defined by 
\begin{equation}\label{def-kappa}
    \kappa_{n,s}:= 2^{-2s} \frac{ \Gamma (\frac{n}{2}-s)} {\Gamma(s)}\pi ^ {-\frac{n}{2} }
\end{equation}
In the distributional sense, we have
$$ (-\Delta)^s \Phi_s=\delta_0\qquad \text{in }\R^n,$$
where $\delta_0$ is a Dirac delta at the origin. More precisely, for any $f\in C^\infty_0(\R^n)$ it holds
$$ (-\Delta)^s(\Phi_s*f)=f\qquad \text{ in }\R^n.$$

Given $r>0$ and $x_0\in\R^n$, for any $ x \in B_r(x_0)$ and any $ y \in \R^n \setminus \overline{B}_r(x_0)$, the Poisson kernel for $(-\Delta)^s$ in the ball $B_r(x_0)$, denoted by $K_{r,x_0}^s$, is the following (see \cite[Ch. I, Def (1.6.11')]{Landkof})
	\begin{equation}
	 K^s_{r,x_0}(x,y) :=  a_{n,s} \Bigg (\frac {r^2-|x-x_0|^2}{|y-x_0|^2-r^2}\Bigg)^s \frac {1}{|x-y|^n}. 
	\end{equation}
  where   $a_{n,s}$ is a normalization constant depending only on $n$ and $s$, given by
\begin{equation}\label{def-cost-a}
a_{n,s} := \Gamma\left(\frac{n}{2}\right)\pi^{-n/2 -1} \sin(\pi s).
\end{equation}
We will clarify the details of this normalization in what follows.  When the center is the origin, we omit it in the notation, writing just $B_r$ and $K_{r}^s$. 
As in the local case, Poisson kernel $K_{r}^s$ gives the explicit solution of the homogeneous Dirichlet problem for the fractional Laplacian in the ball $B_r$.
Specifically, given $r>0$ and $g \in L^1_s(\R^n) \cap C({\R^n})$ the function defined by 
\begin{equation}
		 u(x) : =
			\begin{cases}	
				\displaystyle  \int_{B_r^c} K^s_r(x,y) g(y)\, dy &\quad  \, \text{if } x\in B_r, \\
				g(x) &\quad \, \text{if } x \in {\obal}
			\end{cases} \label{solD}
	\end{equation}
belongs to $C^\infty(B_r)\cap C(\overline{B}_r)$ and it solves
	\begin{equation*}
	\begin{cases}
	\frlap u= 0 \qquad  &\mbox{ in }  {B_r},
\\	u= g \qquad  &\mbox{ in }  {\obal},
		\end{cases}
	\end{equation*}
    see once again \cite[Ch. I,(1.6.19) p.125]{Landkof}. 
Moreover, from the Poisson kernel in a ball, it is possible to derive the companion mean value property for $\frlap$. Given $r>0$, let us set the function $A^s_r$ as
	\begin{equation} \label{smeandefn}
	A^s_r(y) :=  \begin{cases}
	K^s_r(0,y)=a_{n,s}  \frac{r^{2s}}{(|y|^2-r^2)^s|y|^n} \quad &y \in \R^n \setminus \overline{B}_r,\\
		0 \quad &y \in \overline B_r.
		\end{cases}
	\end{equation}
By choice of the constant $a_{n,s}$ in \eqref{def-cost-a}, for any $r>0$
\begin{equation*}  \int _{\R^n} A^s_r(y) \, dy =1,
\end{equation*}
see \cite[Lemma A.4]{bucur2016some}. The kernel $A^s_r$ defined in \eqref{smeandefn} is used to state the $s$-mean value property, which makes it reasonable to say that $A^s_r$ plays the role of the $s$-mean kernel. Namely, the following propositions hold, see respectively \cite[Ch I, Section 6, n.20]{Landkof} and \cite[Ch. I, Section 6, n.24, Corollary p.126]{Landkof} or \cite[Theorem 2.2]{bucur2016some}.
\begin{Proposition}[$s$-mean value property]\label{prop-mean-1}
Let $s\in(0,1)$. Let $\Omega\subseteq\R^n$ be an open set and $u\in  C_{\textrm{loc}}^{2s+\eee}\cap L^1_s(\R^n)$ be a $s$-harmonic (respectively $s$-subharmonic or $s$-superharmonic) function in $\Omega$. Then for any $x\in\Omega$ 
\begin{equation}
\label{mean-prop}
    u(x)=(\text{resp.}\leq \text{ or }\geq)\,A^s_r*u(x)=\int_{B_r^c}K^s_r(0,y)u(x-y)\,dy
\end{equation}
for any $r>0$ arbitrarily small. Conversely, if $u$ satisfies condition \eqref{mean-prop} in $x\in\R^n$ for $r>0$ arbitrarily small, then $\frlap u(x)=0$ (respectively $\leq 0 $ or $\geq 0$).
\end{Proposition}

\begin{Proposition}\label{prop-mean-2}
Let $s\in(0,1)$. Let $\Omega\subseteq\R^n$ be an open set and $u\in C_{\textrm{loc}}^{2s+\eee}(\R^n)\cap L^1_s(\R^n)$ be a $s$-subharmonic (respectively $s$-superharmonic) function in $\Omega$. Fixing $x\in\Omega$, the map 
\begin{equation}\label{def-M}
    r\mapsto M_s(u,r)(x):=A^s_r*u(x)
\end{equation}
is nondecreasing (respectively nonincreasing) for $r>0$ arbitrarily small.
\end{Proposition}

In addition, we recall  the  nonlocal integration by parts formula (the Green's second identity) and the divergence-type theorem for the fractional Laplacian, these formulas were established in \cite{dipierro2017nonlocal}. 
For any $f,g \in C^{2}(\R^n)$ bounded and $D \subset \R^n$ open we have that 

\begin{equation} \label{eq:parts}
\frac{C_{n,s}}{2} \int_{\R^{2n} \setminus (D^c)^2} \frac{(f(x)-f(y))(g(x)-g(y))}{|x-y|^{n+2s}} \, dx \, dy = \int_{D} g(x) (-\Delta)^s f(x) \, dx + \int_{D^c} g(x) \mathcal{N}_s^{D} f(x) \, dx
\end{equation}
where 
\[
\N_s^D f(x) := \int_{D} \frac{f(x)-f(y)}{|x-y|^{n+2s}} \, dy.
\]
Directly from \eqref{eq:parts} exchanging the order of $f$ and $g$ we also get
\begin{equation} \label{eq:Green}
    \int_{D} g(x)[-(-\Delta)^s f(x)] - f(x) [-(-\Delta)^s g(x)] \, dx = \int_{D^c} g(x) \N_s^D f(x) - f(x) \N_s^D g(x) \, dx .
\end{equation}
Moreover for $g=1$ in \eqref{eq:parts} we retrieve 
\begin{equation}\label{div-nonloc}
\int_{D} [-(-\Delta)^s f(x)] \, dx = \int_{D^c} \N_s^{D} f(x) \, dx.   
\end{equation}

\section{Main results} \label{sec:main_result}

Let $s\in(0,1)$ and $\eee>0$. Let  $u\in C_{\textrm{loc}}^{s+\eee}(\R^n)\cap L^2_s(\R^n)$. We define the following functional
\begin{equation} \label{eq:ACFnonlocale}
    J_{ACF}^{s}(u,R) := \frac{n \omega_n }{2R^{1+s}}\int_{0}^{R} r^s \int_{B_r^c} K^s_r(0,y)G_u(y)\, dy \, dr=\frac{n \omega_n }{2R^{1+s}}\int_{0}^{R} r^s M_s(G_u,r)(0)\,dr
    \end{equation}
with $R>0$, where $K^s_r$ is the Poisson kernel defined in \eqref{poissondefn}, $G_u$ is the quantity defined in \eqref{def-G} and $M_s$ is the $s$-mean operator of \eqref{def-M}. Expanding all the terms in \eqref{eq:ACFnonlocale}, the functional $J^s_{ACF}$ reads as
\begin{equation} \label{eq:ACFnonlocale-ex}
    J_{ACF}^{s}(u,R) = \frac{ n
    \omega_n C_{n,s} \, a_{n,s}}{2R^{1+s}} \int_{0}^{R} r^s \int_{B_r^c} {\frac{r^{2s}}{( \nor{y}^2 -r^2)^{^s} \nor{y}^n} \left(  \int_{\R^n} \frac{(u(y)-u(\eta))^2}{\nor{y-\eta}^{n+2s}}  \, d\eta\right) \,} dy \, dr.
\end{equation}

\begin{Remark}
    Given $\varepsilon >0$, for any $u \in C_{\rm loc}^{s+\varepsilon}(\R^n) \cap L_s^2(\R^n)$, it is possible to rewrite the functional $J_{ACF}^s$ as
\begin{equation}\label{J_ACF_inv}
J_{ACF}^s(u,R) = \frac{ n
    \omega_n a_{n,s}}{2R^{1+s}} \int_{0}^{R} r^s \int_{B_r} 
\frac{1}{\left( r^2- |x|^2 \right)^s |x|^{\,n-2s}} 
G_u\left( \frac{r^2}{|x|^2} x \right) \, dx \, dr,
\end{equation}
for $R>0$. 

Indeed, applying the Kelvin transform $x = \frac{r^2}{|y|^2} y$ to \eqref{eq:ACFnonlocale}, we obtain
\begin{align*}
J_{ACF}^s(u,R) 
&= \frac{ n
    \omega_n a_{n,s}}{2R^{1+s}} \int_{0}^{R} r^s \int_{B_r^c} 
\frac{r^{2s}}{\left( |y|^2 - r^2 \right)^s |y|^n} G_u(y) \, dy \, dr \\
&= \frac{ n
    \omega_n C_{n,s}a_{n,s}}{2R^{1+s}} \int_{0}^{R} r^s \int_{B_r} 
\frac{r^{2s}}{\left( r^4 |x|^{-2} - r^2 \right)^s |x|^{n}} 
\left( \int_{\R^n} \frac{\bigl( u\bigl(\frac{r^2}{|x|^2} x\bigr) - u(\eta) \bigr)^2}{|\frac{r^2}{|x|^2}x - \eta|^{\,n+2s}} \, d\eta \right) dx \, dr \\
&= \frac{ n
    \omega_n a_{n,s}}{2R^{1+s}} \int_{0}^{R} r^s \int_{B_r} 
\frac{1}{\left( r^2 - |x|^2 \right)^s |x|^{\,n-2s}} 
G_u\left( \frac{r^2}{|x|^2} x \right) \, dx \, dr,
\end{align*}
which concludes the proof of \eqref{J_ACF_inv}.
\end{Remark}

With the representation formula given by \eqref{J_ACF_inv}, it is not difficult to show that the asymptotic profile as $s \to 1^{-}$ of the mean value kernel converges to the local one.
\begin{Lemma}\label{lemma:Asymp2}
    For any $g \in C_{\textrm{loc}}^{2s + \varepsilon}(\R^n) \cap L_s^1(\R^n)$ it holds
    \begin{equation*}
        \lim_{s\to 1^-}  a_{n,s} \int_{B_r} \frac{1}{(r^2-\nor{y}^2)^s} \frac{1}{\nor{y}^{n-2s}} g(y) \, dy = \fint_{\p B_r} g(y) \, d\mathcal{H}^{n-1}(y)
    \end{equation*}
\end{Lemma}
\begin{proof} By the Coarea formula, we reach 
    \begin{align*}
         a_{n,s} \int_{B_r} \frac{1}{(r^2-\nor{y}^2)^s}& \frac{1}{\nor{y}^{n-2s}} g(y) \, dy \\
       &=  a_{n,s} \int_{0}^{r} \left( \int_{\p B_t(0)}\frac{1}{(r^2-t^2)^s} \frac{1}{t^{n-2s}} g(y) \, d\mathcal{H}^{n-1}(y) \right) \, dt \\
        &=  a_{n,s} \int_{0}^{r} \left( \int_{\p B_t(0)}\frac{1}{r^{2s}(1-\frac{t^2}{r^2})^s} \frac{1}{t^{n-2s}} g(y) \, d\mathcal{H}^{n-1}(y) \right) \, dt. 
        \end{align*} 
        Letting $\tau = \frac{t}{r} $ and defining \[h^s(r,\tau):=r^{1-n}(1+\tau)^{-s} \tau^{2s-n}\int_{\p B_{r\tau}(0)} g(y)d\mathcal{H}^{n-1}(y), \] we get 
         \begin{align*}
            a_{n,s} \int_{B_r}& \frac{1}{(r^2-\nor{y}^2)^s} \frac{1}{\nor{y}^{n-2s}} g(y) \, dy \\
        &=   a_{n,s} \int_{0}^{1} \frac{1}{r^{n-1}(1-\tau^2)^s\tau^{n-2s}} \left( \int_{\p B_{r\tau}(0)} g(y) \, d\mathcal{H}^{n-1}(y) \right) \, d\tau \\
        &=  a_{n,s} \int_{0}^{1} \frac{1}{(1-\tau)^s} h^s(r,\tau) \, d\tau   \\&= \frac{ a_{n,s}}{1-s} \lim_{\varepsilon \to 0} \left\{ \left[ -(1-\tau)^{1-s} h^s(r,\tau) \right]_{\varepsilon}^1 + \int_{\varepsilon}^1 (1-\tau)^{1-s} \p_{\tau} h^s(r, \tau) \, d\tau \right\}
    \end{align*}
    Since $\frac{ a_{n,s}}{1-s} \to \frac{2}{n \omega_n}$ as $s \to 1^{-}$, taking the limit, we end up with
\begin{align*}
    \frac{2}{n \omega_n} \lim_{\varepsilon \to 0}\left[ h^1(r,\varepsilon) + \int_{\varepsilon}^1 \p_{\tau} h^1(r,\tau) \, d\tau\right]= \frac{1}{n \omega_n} h^{1}(r,1) =  \fint_{\p B_r} g(y) \, d\mathcal{H}^{n-1}(y).
\end{align*}
\end{proof}

The same result can be obtained starting from the representation with the exterior ball of \eqref{eq:ACFnonlocale-ex}, see for instance \cite[Appendix C]{abatangelo2015large}.

The first remarkable property of $J^s_{ACF}$ that we provide is the following scaling invariance, which mimics the invariance of the local ACF functional under linear dilatation. This motivates also the choice of the normalizing power that appears in the functional.
\begin{Proposition}\label{prop:scaling}
   Let $u\in C_{\textrm{loc}}^{s+\eee}(\R^n)\cap L^2_s(\R^n)$ and $R>0$ be fixed. The functional $J_{ACF}^{s}$ is invariant with respect to the scaling given by  $$u_{\lambda}(x)=\frac{1}{\lambda^s} u(\lambda x).$$Namely
\[ 
J_{ACF}^{s}\left(u_\lambda,\frac{R}{\lambda}\right)= J_{ACF}^{s}(u,R).
\]
\end{Proposition}
\begin{proof}
First, we notice the following scaling property of $G_u$
\begin{align*}
    G_{u_\lambda}(y)&= C_{n,s}    \int_{\R^n} \frac{(u_\lambda(y)-u_\lambda(\eta))^2}{\nor{y-\eta}^{n+2s}}  \, d\eta 
    = C_{n,s}  \int_{\R^n}  \lambda^{-2s} \frac{(u(\lambda y)-u( \lambda 
    \eta))^2}{\nor{y-\eta}^{n+2s}}  \, d\eta \\
    &=  C_{n,s}  \int_{\R^n}  \lambda^{-n-2s} \frac{(u(\lambda y)-u( \xi))^2}{\nor{y-\frac{\xi}{\lambda}}^{n+2s}}  \, d\xi 
    = C_{n,s}  \int_{\R^n} \frac{(u(\lambda y)-u( \xi))^2}{\nor{\lambda y-\xi}^{n+2s}}  \, d\xi =G_{u}(\lambda y),
\end{align*}
where in the third equality we used the change of variable $\lambda\eta=\xi$. As a consequence, we get
\begin{align*}
J_{ACF}^{s}\left(u_\lambda,\frac{R}{\lambda}\right) &= \frac{ n
    \omega_n \, a_{n,s} \lambda^{1+s}}{2 R^{1+s}} \int_{0}^{R/\lambda} r^s \int_{B_r^{c}(0)} \frac{r^{2s}}{( \nor{y}^2 -r^2)^{s}\nor{y}^n}\, G_u(\lambda y) \, dy \, dr \\
    &=\frac{  n
    \omega_n \, a_{n,s} \lambda^{-2s}}{2 R^{1+s}} \int_{0}^{R} \rho^s \int_{B_{\rho/\lambda}^c} \frac{\rho^{2s}}{(\nor{y}^2- \lambda^{-2} \rho^2)^{s} \nor{y}^n}\,G_u(\lambda y) \, dy \, d\rho \\
    &= \frac{  n
    \omega_n\, a_{n,s} \lambda^{-2s}}{2R^{1+s}} \int_{0}^{R} \rho^s \int_{B_{\rho}^c} \frac{\rho^{2s}}{\lambda^{-2s}(\nor{x}^2- \rho^2)^{s} \lambda^{-n} \nor{x}^n}\, G_u(x)  \lambda^{-n}\, dx \, d\rho \\
    &= \frac{  n
    \omega_n\, a_{n,s} }{2R^{1+s}} \int_{0}^{R} \rho^s \int_{B_{\rho}^c} \frac{\rho^{2s}}{(\nor{x}^2- \rho^2)^s \nor{x}^n}\, G_u(x) \, dx \, d\rho = J_{ACF}^{s}\left(u,R\right),
\end{align*}
where in the second equality we used the change of variable $\rho= \lambda r$ and in the third equality the change $x=\lambda y$. This concludes the proof. 
\end{proof}

The main result of this section concerns the following nonlocal ACF-type  monotonicity formula, under the $s$-subharmonicity assumption on $G_u$.

\begin{Remark}
   
Let us briefly explain the motivation behind the assumptions of Theorem \ref{Thm:sACF} and Theorem \ref{Thm:ACF-lim} (although they might be further relaxed).  First, the condition \( u \in C_{\textrm{loc}}^{s+\varepsilon}(\mathbb{R}^n) \cap L^2_{s}(\mathbb{R}^n) \) ensures that the quantity \( G_u \) is well defined, controlling both its behavior near the singularity (via local H\"older continuity) and at infinity (via integrability). The requirement \( G_u \in C_{\mathrm{loc}}^{2s+\delta}(\mathbb{R}^n) \cap L_s^1(\mathbb{R}^n) \) guarantees the well‐posedness of the fractional Laplacian \( (-\Delta)^s G_u \). Moreover, if $u \in C_{\textrm{loc}}^{1,\gamma}(\R^n) \cap L_{s_0}^2(\R^n)$ for some $\gamma>0$, ensures the uniform convergence \( G_u \to 2|\nabla u|^{2} \) as \( s \to 1^- \), see Lemma \ref{lemma:Asymp1}.

\end{Remark}

We are ready to introduce the proof of Theorem \ref{Thm:sACF} in the case $g_u=\frac{1}{2}G_u$.

\begin{customproof}{Theorem \ref{Thm:sACF}}

For every $R >0 $ and $\lambda \in (0,1)$, using the notation of \eqref{eq:ACFnonlocale}, we obtain
\begin{align*}
    J_{ACF}^s(u,R)&=\frac{ n
    \omega_n}{2R^{1+s}} \int_{0}^{R} r^s \int_{B_r^{c}(0)}  K_{r}^s (0,y)  G_{u}(y)  \, dy \, dr  \\ 
    &\underbrace{=}_{\lambda r= \rho } \frac{ n
    \omega_n a_{n,s}}{2R^{1+s}}   \int_{0}^{\lambda R} \frac{1}{\lambda} \left( \frac{\rho}{\lambda} \right)^s \int_{B_{\frac{\rho}{\lambda}}^{c}(0)}  \frac{ \left(\frac{\rho }{\lambda}\right)^{2s}}{\left( \nor{y}^2- \frac{\rho^2}{\lambda^2} \right)^s}  \frac{1}{\nor{y}^n} G_{u}(y)  \, dy \, d\rho \\ 
    &= \frac{ n
    \omega_n a_{n,s}}{2(\lambda R)^{1+s}} \int_{0}^{\lambda R} \rho^s \int_{B_{\frac{\rho}{\lambda}}^{c}(0)}  \frac{ \rho^{2s}}{\left( \nor{y}^2- \rho^2 \right)^s}  \frac{1}{\nor{y}^n} G_{u}(y)  \, dy \, d\rho \\
    & = \frac{ n
    \omega_n}{2(\lambda R)^{1+s}} \int_{0}^{\lambda R} \rho^s  M_s\left(G_u, \frac{\rho}{\lambda}\right)(0) \, d\rho \\
    & \geq  \frac{ n
    \omega_n}{2(\lambda R)^{1+s}} \int_{0}^{\lambda R}  \rho^s  M_s\left(G_u, \rho\right)(0) \, d\rho= J_{ACF}^s(u,\lambda R)
\end{align*}
where, in the last inequality, we use Proposition \ref{prop-mean-2} having $\frac{\rho}{\lambda} > \rho$.
That is for every $R_1 \leq R_2$ 
\begin{equation}
    J_{ACF}^s(u,R_1) \leq J_{ACF}^s(u,R_2).
\end{equation} \end{customproof}

We now present a stability result for the functional $J_{ACF}^{s}(u,R)$ to its local counterpart as $s\to1^-$.
\begin{Theorem}\label{Thm:ACF-lim}
Let  $s_0 \in (0,1)$. Assume that $u \in C_{loc}^{1,1}(\R^n) \cap L_{s_0}^2(\R^n)$ and $
 G_u \in C_{loc}(\R^{n}) \cap L_{s_0}^1(\R^n)$, for $R>0$ it holds
 \begin{equation}
    \lim_{s \to 1^-} J_{ACF}^{s}(u,R) = J_{ACF}(u,R)=\frac{1}{R^2}\int_{B_R} \frac{|\nabla u|^2 }{\nor{y}^{n-2}} \, dy.
 \end{equation} 
\end{Theorem}

\begin{proof}
Let us consider $s\in [s_0,1)$. Regarding the limit as $s \nearrow 1^{-}$, we write $J_{ACF}^s(u,R)$ as
    \begin{equation*}
        J_{ACF}^s(u,R) = \frac{ n
    \omega_n}{2R^{1+s}} \int_{0}^{R} r^s M_s(G_u,r)(0) \, dr.
    \end{equation*}

Let $R_0 > \max \{\sqrt{2}R, 1\}$ and consider $r \in (0, R]$. Then, we have 
\begin{equation} \label{eq:split_2r}
M_s(G_u, r)(0) = \int_{B_{R_0} \setminus B_r} K_r^s(0,y) G_u(y) dy + \int_{B_{R_0}^c} K_r^s(0,y) G_u(y)dy=:I_1+I_2
\end{equation}
Let us first consider $I_2$. For any $y \in B_{R_0}^c$, we have $|y| > R_0  > \sqrt{2}R$, which implies
\begin{equation*}
|y|^2 - r^2 = |y|^2 \left( 1 - \frac{r^2}{|y|^2} \right) \ge  \frac{1}{2} |y|^2.
\end{equation*}
Consequently, the Poisson kernel $K^s_r(0,y)$ satisfies the following uniform pointwise bound on $ B_{R_0}^c$
\begin{equation*}
K_r^s(0,y) = a_{n,s} \frac{r^{2s}}{( |y|^2 - r^2 )^s |y|^n} \le 2^sa_{n,s}  \frac{r^{2s}}{|y|^{n+2s}} \le 2a_{n,s} \frac{r^{2s}}{|y|^{n+2s}}.
\end{equation*}
Since $|y| > R_0 > 1$  and $s \in [s_0, 1)$, there exists a constant $C = C(n, s_0, R_0) > 0$ such that $$\frac{r^{2s}}{|y|^{n+2s}} \le \frac{C}{(1+|y|)^{n+2s_0}}$$ for all $r \le R$. Thus, we find 
\begin{equation} \label{eq:tail_bound_clean}
\left| \int_{B_{R_0}^c} K_r^s(0,y) G_u(y)dy \right| \le 2 C a_{n,s} \int_{B_{R_0}^c} \frac{|G_u(y)|}{(1+|y|)^{n+2s_0}} dy \le 2C a_{n,s} \|G_u\|_{L_{s_0}^1(\mathbb{R}^n)}.
\end{equation}
Recalling that  $G_u \in L_{s_0}^1(\mathbb{R}^n)$, $a_{n,s} \to 0 $ as $s\to 1^-$ by  \eqref{def-cost-a}, and by \eqref{eq:tail_bound_clean} we conclude 
\begin{equation} \label{eq:lim_tail_2r}
\lim_{s \to 1^-} \int_{ B_{R_0}^c} K_r^s(0,y) G_u(y) dy = 0.
\end{equation}    

Let us consider now the term $I_1$. By adding and subtracting  $2|\nabla u(y)|^2$ in $I_1$, we have
\begin{equation} \label{eq:local_split_2r}
I_1= \int_{B_{R_0} \setminus B_r} K_r^s(0,y) \left( G_u(y) - 2 |\nabla u(y)|^2 \right) dy + 2 \int_{B_{R_0} \setminus B_r} K_r^s(0,y) |\nabla u(y)|^2 dy.
\end{equation}   
By the uniform convergence of $\frac{1}{2}G_u$ to $|\nabla u|^2$ on compact sets as $s \to 1^-$ of Lemma \ref{lemma:Asymp1}, recalling \eqref{smeandefn}, we obtain
\begin{equation} \label{eq:lim_diff_2r}
\left| \int_{B_{R_0} \setminus B_r} K_r^s(0,y) \left( G_u(y) - 2 |\nabla u(y)|^2 \right) dy \right| \le \sup_{y \in B_{R_0} \setminus B_r} \left| G_u(y) - 2 |\nabla u(y)|^2 \right| \longrightarrow 0 \quad \text{as } s \to 1^-.
\end{equation} 

Let us consider $g(y) := |\nabla u(y)|^2\chi_{B_{R_0}}(y)$, by the proof in \cite[Appendix C]{abatangelo2015large}, we conclude 
\begin{equation} \label{eq:lim_grad_2r}
\lim_{s \to 1^-} \int_{B_{R_0} \setminus B_r} K_r^s(0,y) |\nabla u(y)|^2 dy = \lim_{s \to 1^-} \int_{ B_r^c} K_r^s(0,y) g(y) dy = \fint_{\partial B_r} |\nabla u(y)|^2  d\mathcal{H}^{n-1}(y).
\end{equation}  
Combining \eqref{eq:lim_tail_2r}, \eqref{eq:local_split_2r}, \eqref{eq:lim_diff_2r}, and \eqref{eq:lim_grad_2r}, for every $r \in (0, R]$, we establish
\begin{equation} \label{eq:pointwise_limit_2r}\lim_{s \to 1^-} M_s(G_u, r)(0) = 2 \fint_{\partial B_r} |\nabla u(y)|^2 d\mathcal{H}^{n-1}(y).\end{equation}

Multiplying \eqref{eq:pointwise_limit_2r} by $r^s$, we have \begin{equation*}
\lim_{s \to 1^-} r^s M_s(G_u, r)(0) = 2 r \fint_{\partial B_r} |\nabla u(y)|^2 d\mathcal{H}^{n-1}(y) = \frac{2}{n \omega_n} \int_{\partial B_r} \frac{|\nabla u(y)|^2}{r^{n-2}} d\mathcal{H}^{n-1}(y).\end{equation*}
Recalling that $G_u \in C_{\mathrm{loc}}(\mathbb{R}^n)$, by \eqref{eq:split_2r}, \eqref{eq:tail_bound_clean} and \eqref{def-cost-a}, we have  
$$
M_s(G_u, r)(0) \le \|G_u\|_{L^\infty(B_{R_0})} + 2 C \, \Gamma\left(\frac{n}{2}\right)\pi^{-n/2 -1}  \|G_u\|_{L_{s_0}^1(\mathbb{R}^n)} < \infty
$$
uniformly for all $r \in (0, R]$ and $s \in [s_0, 1)$. Thus, by the Dominated Convergence Theorem, we obtain \begin{equation*}\lim_{s \to 1^-} \int_0^R r^s M_s(G_u, r)(0) dr = \frac{2}{n \omega_n} \int_0^R \int_{\partial B_r} \frac{|\nabla u(y)|^2}{r^{n-2}}\, d\mathcal{H}^{n-1}(y) \, dr.\end{equation*}
Finally, applying the Coarea formula, we  conclude
\begin{align*}\lim_{s \to 1^-} J_{ACF}^s(u,R) &= \frac{n \omega_n}{2R^2} \cdot \frac{2}{n \omega_n} \int_0^R \! \int_{\partial B_r} \frac{|\nabla u(y)|^2}{r^{n-2}} d\mathcal{H}^{n-1}(y) dr \\ &= \frac{1}{R^2} \int_{B_R} \frac{|\nabla u(y)|^2}{|y|^{n-2}} dy = J_{ACF}(u,R).
\end{align*}
\end{proof}

In the next result, we provide an estimate for $J_{ACF}^s(u,R)$ in terms of another functional that mirrors the structure of the local Alt–Caffarelli–Friedman monotonicity formula \eqref{ACF-intro}. The appearance of the fundamental solution of the fractional Laplacian in this estimate plays a decisive role, as it is the key ingredient that allows one to derive interior nonlocal gradient bounds.

\begin{Proposition}\label{stima-tipo-locale} Let  $s \in (0,1)$ and $\varepsilon>0 $ be a small fixed constant. Assume that $u \in C_{loc}^{s+\varepsilon}(\R^n) \cap L_{s}^2(\R^n)$. Then for $R>0$, it holds
\begin{equation} \label{stimalocale}
    J_{ACF}^{s}(u,R) \leq \frac{C}{R^{2s}} \int_{\R^{n}} \frac{G_u(x)}{|x|^{n-2s}} \, dx,
\end{equation}
    where $C$ is a constant depending only on $n$ and $s$.
\end{Proposition}
\begin{proof}
    For $R >0$, by Tonelli's Theorem, we have
\begin{equation}\label{uso-Tonelli}
\begin{split}
  J_{ACF}^{s}(u,R)&= \frac{n\omega_n a_{n,s}}{2R^{1+s}} \left( \int_0^R r^s \int_{B_r^c} \frac{r^{2s}}{(|x|^2-r^2)^s} \frac{G_u(x)}{|x|^n} \, dx \, dr  \right)\\
    &= \frac{n\omega_n a_{n,s}}{2 R^{1+s}} \left(  \int_{B_R^c} \frac{G_u(x)}{|x|^n} \int_{0}^R \frac{r^{3s}}{(|x|^2 -r^2)^s} \, dr \, dx + \int_{B_R} \frac{G_u(x)}{|x|^{n}} \int_{0}^{|x|} \frac{r^{3s}}{(|x|^2-r^2 )^s} \, dr \, dx \right) \\
    &= \frac{n\omega_n a_{n,s}}{2 R^{1+s}} \left(  \int_{B_R^c} \frac{G_u(x)}{|x|^{n+2s}} \int_{0}^R \frac{r^{3s}}{\left(1 -\frac{r^2}{|x|^2}\right)^s} \, dr \, dx + \int_{B_R} \frac{G_u(x)}{|x|^{n+2s}} \int_{0}^{|x|} \frac{r^{3s}}{\left(1 -\frac{r^2}{|x|^2}\right)^s} \, dr \, dx \right)
     \\ &=  \frac{n\omega_n a_{n,s}}{2 R^{1+s}} \left( \int_{B_R^c} \frac{G_u(x)}{|x|^{n+2s}} \int_{0}^{\frac{R}{|x|}} \frac{|x|^{3s+1} \tau^{3s}}{\left(1 -\tau^2\right)^s} \, d\tau \, dx + \int_{B_R} \frac{G_u(x)}{|x|^{n+2s}} \int_{0}^1 \frac{|x|^{3s+1} \tau^{3s}}{\left(1 -\tau^2 \right)^s} \, d\tau \, dx \right) \\
     &=  \frac{n\omega_n a_{n,s}}{2 R^{1+s}} \left(  \int_{B_R^c} \frac{G_u(x)}{|x|^{n-s-1}} \int_{0}^{\frac{R}{|x|}} \frac{ \tau^{3s}}{\left(1 -\tau^2\right)^s} \, d\tau \, dx + \int_{B_R} \frac{G_u(x)}{|x|^{n-s-1}} \int_{0}^1 \frac{ \tau^{3s}}{\left(1 -\tau^2 \right)^s} \, d\tau \, dx \right)
     \end{split}
    \end{equation}
    For any $s \in (0,1)$, it holds that $2s-n < s+1-n$, thus $|x|^{s+1-n} \leq |x|^{2s-n}$ for $|x| \leq 1$.
Consequently, the second term of \eqref{uso-Tonelli} can be further estimated as
\begin{align*}
     \int_{B_R} \frac{G_u(x)}{|x|^{n-s-1}} \, dx &= R^{1+s} \int_{B_1} \frac{G_u(Ry)}{|y|^{n-s-1}} \, dy  \\
     &\leq  R^{1+s} \int_{B_1} \frac{G_u(Ry)}{|y|^{n-2s}} \, dy = R^{1-s} \int_{B_R} \frac{G_u(x)}{|x|^{n-2s}} \, dx.
\end{align*}
This leads to the following bound for $J_{ACF}^{s}(u,R)$.
\begin{equation} \label{eq:stima3}
\begin{aligned}
    J_{ACF}^{s}(u,R) \leq \frac{n\omega_n a_{n,s}}{2R^{1+s}} \left( R^{1+s}\int_{B_1^c} \frac{G_u(R y)}{|y|^{n-s-1 }}  \int_{0}^{|y|^{-1}}\hspace{-1em}\frac{\tau^{3s}}{(1-\tau^{2})^s} \, d\tau \, dy \right. \\ + \left. R^{1-s} \int_{0}^{1} \frac{\tau^{3s}}{(1-\tau^2)^s} \, d\tau \int_{B_R} \frac{G_u(x)}{|x|^{n-2s}} \, dx\right).
    \end{aligned}
\end{equation}
Defining $$ \rho \mapsto Q(\rho):= \int_{0}^{\rho} \frac{\tau^{3s}}{(1-\tau^2)^s} \, d\tau,$$
we get
\[
\lim_{\rho\to 0}\frac{Q(\rho)}{\rho^{1-s}} = \lim_{\rho\to0} \frac{\rho^{3s}}{(1-\rho^2)^s} \frac{(1-s)^{-1}}{\rho^{-s}}=\lim_{\rho\to0} \frac{\rho^{4s}}{(1-\rho^2)^s} \frac{1}{1-s}=0.
\]
Thus the function $Q(\rho)\rho^{s-1} $ is bounded by a constant depending only on $n$ and $s$ for all $ 0< \rho\leq 1$ since it is continuous. We use this fact to deal with the first term in \eqref{eq:stima3}. Indeed,
\begin{align*}
     R^{1+s}\int_{B_1^c} &\frac{G_u(R y)}{|y|^{n-s-1}} \int_{0}^{|y|^{-1}}\hspace{-1em}\frac{\tau^{3s}}{(1-\tau^{2})^s} \, d\tau \, dy =   R^{1+s}\int_{B_1^c} \frac{G_u(R y)}{|y|^{n-s-1}} \left(Q(|y|^{-1})|y|^{1-s} \right) |y|^{s-1} \, dy \\
     &\leq C R^{1+s} \int_{B_1^c}\frac{G_u(R y)}{|y|^{n-2s}} \, dy = C R^{1-s} \int_{B_R^c}\frac{G_u(x)}{|x|^{n-2s}} \, dx.
\end{align*}

Hence, combining this with \eqref{eq:stima3}, we obtain
\begin{equation}
\begin{aligned} \label{eq:stima4}
    J_{ACF}^s(u,R) &\leq \frac{n\omega_n a_{n,s}}{2 R^{1+s}} \left(  C R^{1-s} \int_{B_R^c}\frac{G_u(x)}{|x|^{n-2s}} \, dx +  R^{1-s} \int_{0}^{1} \frac{\tau^{3s}}{(1-\tau^2)^s} \, d\tau \int_{B_R} \frac{G_u(x)}{|x|^{n-2s}} \, dx \right)  \\
    &\leq \frac{C}{R^{2s}} \int_{\R^{n}} \frac{G_u(x)}{|x|^{n-2s}} \, dx
\end{aligned}
\end{equation}
where $C$ is a constant depending only on $n$ and $s$.

\end{proof}

The next lemma is an application of the use of the  functional $J_{ACF}^s(u,R)$. The bound obtained in Proposition \ref{stima-tipo-locale} together with the $s$-subharmonicity of $G_u$ implies interior \textit{nonlocal} gradient estimates. Interior gradient estimates for nonlocal equations have been obtained in \cite{cabre2022bernstein}.

\begin{Lemma} \label{Lm:IntGradientEst_in0}
  Let $s\in (0,1)$. Let $u\in C^{2}_{loc}(\R^n)\cap L^2_s(\R^n)$, be such that 
    \begin{equation*}
         \begin{cases}
        (-\Delta)^s u = f \quad &\textit{in } B_{1} \\
        u = 0  \quad &\textit{on } B_{1}^c,
    \end{cases}
    \end{equation*}
    with $f \in L^\infty(B_{1})$. Let us assume that for  $G_u\in C^{2s+\delta}_{loc}(\R^n)\cap L^1_s(\R^n)$ some $\delta>0$ small, and $(-\Delta)^s G_u\leq 0$ in $B_{1}$. Then,
    \begin{equation}\label{clam1}
    G_u(0)\leq \frac{C}{R^{2s}}\left(u^2(0)+\|u\|_{L^\infty(B_{1})}\|f\|_{L^\infty(B_{1})}   \right)
    \end{equation}
    where $0<R<1$, and $C$ is a positive universal constant depending only on $n$ and $s$.
    \end{Lemma}
\begin{proof} Exploiting Lemma \ref{lemma:tec}, we have
$$  [-(-\Delta)^s u^2] + 2u f \, \chi_{B_{1}} = G_u \quad \text{in }\R^n $$ 
Hence, by Proposition \ref{stima-tipo-locale}, we obtain that
\begin{align}\label{eq:stima5}
J_{ACF}^s(u,R) &\leq  \frac{C}{R^{2s}} \left( \int_{\R^n} \frac{[-(-\Delta)^s u^2]}{|x|^{n-2s}} \, dx + 2 \int_{B_{1}} \frac{u f}{|x|^{n-2s}} \, dx \right) \\
&\leq  \frac{C}{R^{2s}} \left( \int_{\R^n} \frac{[-(-\Delta)^s u^2]}{|x|^{n-2s}} \, dx +\frac{n\omega_n}{s}\|u\|_{L^\infty(B_{1})}\|f\|_{L^\infty(B_{1})}   \right),\nonumber
\end{align}
for $0<R<1$.

Now, let us define $F_\delta(x) := \min \{ |x|^{2s-n} , \delta^{2s-n} \}$ for some $\delta>0$. Let $\rho \in C_0^\infty(B_1)$ be a standard mollifier, with $\rho \geq 0$ and 
$\int_{\R^n}\rho = 1$. For $\varepsilon>0$ sufficiently small, let us set
\[
\rho_\varepsilon(x) := \varepsilon^{-n}\rho\!\left(\frac{x}{\varepsilon}\right),\quad  \Gamma_\delta=F_\delta \ast \rho_\varepsilon
\quad \text{and}\quad 
u_\varepsilon := u* \rho_\varepsilon.
\] Our goal will be to estimate the first term in \eqref{eq:stima5} with $\Gamma_\delta$ instead of $|x|^{2s-n}$ uniformly on $\delta$ and $\varepsilon$ and then let $\delta,\varepsilon\to 0$.

By \eqref{eq:Green}, we get
\begin{align*}
    \int_{\R^n} [&-(-\Delta)^s u_\varepsilon^2] \Gamma_\delta  \, dx \\
    &= \int_{B_\delta^c} [-(-\Delta)^s u_\varepsilon^2] \Gamma_\delta \, dx  + \int_{B_\delta} [-(-\Delta)^s u_\varepsilon^2] \Gamma_\delta \, dx \\
    &= \int_{B_\delta^c}  [-(-\Delta)^s \Gamma_\delta] u_\varepsilon^2 \, dx - \int_{B_\delta} u_\varepsilon^2 \N_s^{B_\delta^c} \Gamma_\delta \, dx+ \int_{B_\delta} \Gamma_\delta \N_s^{B_\delta^c} u_\varepsilon^2 \, dx + \int_{B_\delta}[-(-\Delta)^s u_\varepsilon^2] \Gamma_\delta \, dx 
\end{align*}
Since $\Gamma_\delta(x) \geq \Gamma_{\delta}(y)$, for $x \in B_{\delta}$ and $y \in B_{\delta}^c$ we know that
\[ 
\int_{B_\delta} u_\varepsilon^2 \N_s^{B_\delta^c} \Gamma_\delta = \int_{B_\delta} u_\varepsilon^2 \int_{B_{\delta}^c} \frac{\Gamma_\delta(x)-\Gamma_\delta(y)}{|x-y|^{n+2s}} \, dy \geq 0
\]
and, exploiting the fact that $\Gamma_\delta \leq 2\delta^{2s-n}$ in $B_{\delta}$, we get
\begin{align}\label{gamma-d}
\int_{\R^n} [-(-\Delta)^s u_\varepsilon^2] \Gamma_\delta \, dx  &\leq    \int_{B_\delta^c} [-(-\Delta)^s \Gamma_\delta] u_\varepsilon^2 \, dx + 2\delta^{2s-n} \left( \int_{B_{\delta}} \N_s^{B_\delta^c} u_\varepsilon^2 \, dx + \int_{B_\delta}[-(-\Delta)^s u_\varepsilon^2] \, dx   \right).
\end{align}  
Finally by \eqref{div-nonloc} and \eqref{gamma-d}, we obtain
\begin{align}\label{cancell-normali}
 \int_{\R^n} [-(-\Delta)^s u_\varepsilon^2] \Gamma_\delta \, dx &\leq\int_{B_\delta^c} [-(-\Delta)^s \Gamma_\delta] u_\varepsilon^2 \, dx + 2\delta^{2s-n} \left( \int_{B_{\delta}} \N_s^{B_\delta^c} u_\varepsilon^2(x) \, dx  +  \int_{B_{\delta}^c} \N_s^{B_\delta} u_\varepsilon^2(y)\, dy \right) \\
  &= \int_{B_\delta^c} [-(-\Delta)^s \Gamma_\delta] u_\varepsilon^2 \, dx.\nonumber
\end{align}
On the other hand, by dominated convergence theorem, we have
\begin{equation}\label{2-lim}
    \begin{split}
        \lim_{\varepsilon\to 0}\left[\lim _{\delta\to0}\int_{B_\delta^c} (-(-\Delta)^s \Gamma_\delta) u_\varepsilon^2 \, dx \right]&= \lim_{\varepsilon\to 0}\left[\lim _{\delta\to0}\int_{B_\delta^c} -((-\Delta)^s F_\delta\ast \rho_\varepsilon) u_\varepsilon^2 \, dx \right]\\
        &= \lim_{\varepsilon\to 0}\left[\int_{\R^n} -((-\Delta)^s |x|^{-n+2s}\ast \rho_\varepsilon) u_\varepsilon^2 \, dx \right] \\
        &=\lim_{\varepsilon\to 0}\left[\kappa_{n,s}^{-1}\int_{\R^n} - \rho_\varepsilon u_\varepsilon^2 \, dx \right]\leq \kappa_{n,s}^{-1}u(0)^2
        \end{split}
\end{equation}
Consequently, by \eqref{eq:stima5}, \eqref{cancell-normali} and \eqref{2-lim}, we have
\begin{equation}\label{eq:ultima-1}
J_{ACF}^s(u,R) \leq \frac{C}{R^{2s}}\left(\kappa_{n,s}^{-1}u^2(0)+\frac{n\omega_n}{s}\|u\|_{L^\infty(B_{1})}\|f\|_{L^\infty(B_1)}   \right).
\end{equation}
Finally, using the mean value property of Proposition \ref{prop-mean-1} for $G_u$, for $0<R<1$, leads to
\begin{equation*} 
\begin{split}
     G_u(0) &= \frac{1+s}{R^{1+s}}\int_{0}^R \rho^s G_u(0) \, d\rho  \\
     &\leq  \frac{1+s}{R^{1+s}}\int_{0}^R \rho^s \int_{B_{\rho}^c} K_{\rho}^s(0,y) G_u(y) \, dy \, d\rho \leq   (1+s) J_{ACF}^s(u,R) 
     \end{split}
\end{equation*}
which, together with \eqref{eq:ultima-1}, proves \eqref{clam1}.
\end{proof}

 The proof of Theorem \ref{Thm:IntGradientEst} follows immediately, making use of Lemma \ref{Lm:IntGradientEst_in0}.
\begin{customproof}{Theorem \ref{Thm:IntGradientEst}}
    
Let $R=\tfrac{1}{4}$, without loss of generality we can restrict ourselves to consider $\|u\|_{L^\infty(B_1)} + \|f\|_{L^\infty(B_1)} \leq 1$, otherwise we apply our analysis to $\bar u$ defined as 
\[
\bar u := \frac{u}{\|u\|_{L^\infty(B_1)} + \|f\|_{L^\infty(B_1)}}.
\] 
Applying Lemma \ref{Lm:IntGradientEst_in0} leads directly to
\[
G_u(0) \leq C_0,
\]
for some constant $C_0$ depending only on $n$ and $s$. 
Repeating the argument within $B_{1/2}$ gives \eqref{claim2-intro}.

\end{customproof}

  \section{The distributional Riesz fractional gradient}\label{alternativedef}

In this section, we propose an alternative definition of the nonlocal quantity $g_u$ in \eqref{ACF-nonlocal-general}, motivated by recent developments in fractional calculus (see, for instance, \cite{CS,CS19,CS22,CS23,S19}).  
We begin by recalling some basic notions and results concerning the distributional Riesz fractional gradient. Subsequently, we show that most of the analysis carried out in Sections \ref{sec:main_result} also extends to this alternative choice of $g_u=|\nabla^s u|^2$.  

In particular, we will highlight that, under suitable regularity assumptions, the commutativity property between the fractional Laplacian and the distributional fractional gradient naturally yields situations in which the compatibility conditions required for the validity of the monotonicity formula \eqref{ACF-nonlocal-general} are satisfied. This leads also to Corollary \ref{coroll-ACF-csv2}.

\subsection{The operators  {$\nabla^s$} and  {$\div^s$}} 

Now we briefly recall some fundamental definitions and the essential features of the non-local operators $\nabla^s$ and $\div^s$, see \cite{S19, CS19,CS22,CS23} and \cite[Section 15.2]{Pon16}.

Letting $s\in(0,1)$, let us set 
\begin{equation}\label{mu-cs}
\mu_{n, s} 
:= 
2^{s}\, \pi^{- \frac{n}{2}}\, \frac{\Gamma\left ( \frac{n + s + 1}{2} \right )}{\Gamma\left ( \frac{1 - s}{2} \right )}.
\end{equation}
We let
\begin{equation*}
\nabla^{s} f(x) 
:=
\mu_{n, s} \lim_{\eee \to 0^+} \int_{B^c_\eee} \frac{y \, f(x + y)}{|y|^{n + s + 1}} \, dy,
\quad
x\in\R^n,
\end{equation*}
be the \emph{fractional $s$-gradient} of $f\in C_0^{1}(\R^n)$ and, similarly, we let
\begin{equation*}
\div^{s} \varphi(x) 
:= 
\mu_{n, s} \lim_{\eee \to 0^+} \int_{B^c_\eee} \frac{y \cdot \varphi(x + y)}{|y|^{n + s + 1}} \, dy,
\quad
x\in\R^n,
\end{equation*}
be the \emph{fractional $s$-divergence} of $\varphi\in C_0^{1}(\R^n;\R^n)$. 
The non-local operators $\nabla^s$ and $\div^s$ are well defined in the sense that the involved integrals converge and the limits exist. Moreover, since 
\begin{equation*}
\int_{B^c_\eee} \frac{z}{|z|^{n + s + 1}} \, dz=0,
\quad
\text{for any }\eee>0,
\end{equation*}
it is immediate to check that $\nabla^{s}c=0$ for all $c\in\R$ and
\begin{align*}
\nabla^{s} f(x)
&=\mu_{n, s} \int_{\R^{n}} \frac{(y - x)  (f(y) - f(x)) }{|y - x|^{n + s + 1}} \, dy,
\quad
x\in\R^n,
\end{align*}
for all $f\in C_0^{1}(\R^n)$. Analogously, we have
\begin{align*}
\div^{s} \varphi(x) 
&= \mu_{n, s} \int_{\R^{n}} \frac{(y - x) \cdot (\varphi(y) - \varphi(x)) }{|y - x|^{n + s + 1}} \, dy,
\quad
x\in\R^n,
\end{align*}
for all $\varphi\in C_0^{1}(\R^n)$.
From the above expressions, it is not difficult to recognize that, given $f\in C_0^{1}(\R^n)$ and $\varphi\in C_0^{1}(\R^n;\R^n)$, it holds that
\begin{equation*}
\nabla^s f\in L^p(\R^n;\R^n)
\quad\text{and}\quad
\div^s\varphi\in L^p(\R^n)	
\end{equation*}
for all $p\in[1,+\infty]$, see \cite[Corollary 2.3]{CS19}.
Finally, the fractional operators $\nabla^s$ and $\div^s$ are \emph{dual}, in the sense that
\begin{equation*}
\int_{\R^n}f\,\div^s\varphi \,dx=-\int_{\R^n}\varphi\cdot\nabla^s f\,dx
\end{equation*}
for all $f\in C_0^{1}(\R^n)$ and $\varphi\in C_0^{1}(\R^n;\R^n)$, see \cite[Section 6]{S19} and \cite[Lemma 2.5]{CS19}.

Finally, we point out that, in general, the assumptions on $f$ can be relaxed. Indeed, an argument similar to that in Remark \ref{rmk:reg-for-Gu} shows that, for $s \in (0,1)$, the fractional gradient $\nabla^s f$ is pointwise well-defined for any function $f \in C^{s+\varepsilon}(\mathbb{R}^n) \cap L^1_{s/2}(\mathbb{R}^n)$. Moreover, this notation is coherent with the asymptotic behavior of the fractional operators $\nabla^s$ and $\div^s$ when $s\to1^-$ for sufficiently regular functions, see the analysis made in \cite{CS23}.

\begin{Remark} \label{scambio}
Let $s \in (0,1)$ and $u \in C^\infty_0(\mathbb{R}^n)$. By \cite[Proposition 2.1]{CS23}, we know that
\[
\nabla^s u = I_{1-s} \nabla u = \nabla I_{1-s} u,
\]
where, for any $\alpha \in (0,n)$, the Riesz potential of order $\alpha$ is defined by
\[
I_\alpha \varphi(x) := \frac{\mu_{n, 1-\alpha}}{n - \alpha} \int_{\mathbb{R}^n} \frac{\varphi(y)}{|x - y|^{n - \alpha}}\, dy, \quad x \in \mathbb{R}^n,
\]
for $\varphi \in C_0^\infty(\mathbb{R}^n; \mathbb{R}^m)$. 

We recall that, if $\alpha, \beta \in (0,n)$ satisfy $\alpha + \beta < n$, then the semigroup property holds:
\[
I_\alpha (I_\beta\, \varphi) = I_{\alpha + \beta} \,\varphi.
\]
Moreover, for $s \in (0,1)$ and $u \in C^\infty_0(\mathbb{R}^n)$, we have
\[
I_s u(x) = \FF^{-1} \left( |\xi|^{-s} \FF u(\xi) \right)(x), \quad \text{and} \quad I_{2s}^{-1} u(x) = (-\Delta)^s u(x).
\]
Then, for $u \in C^\infty_0(\mathbb{R}^n)$ and any $i = 1, \ldots, n$, we get
\begin{equation}\label{long-four}
\begin{split}
(-\Delta)^s (\partial^s_i u)(x)
&= \FF^{-1} \left( |\xi|^{2s} \FF(\partial^s_i u)(\xi) \right)(x) \\
&= \FF^{-1} \left( |\xi|^{2s} \FF(I_{1-s} \partial_i u)(\xi) \right)(x) \\
&= \FF^{-1} \left( |\xi|^{2s} |\xi|^{s-1} (i\xi_i \FF u(\xi)) \right)(x) \\
&= \FF^{-1} \left( i\xi_i |\xi|^{3s - 1} \FF u(\xi) \right)(x) \\
&= \FF^{-1} \left( i\xi_i |\xi|^{s - 1} \FF((-\Delta)^s u)(\xi) \right)(x) \\
&= \FF^{-1} \left( |\xi|^{s - 1} \FF(\partial_i ((-\Delta)^s u))(\xi) \right)(x) \\
&= I_{1 - s} \partial_i ((-\Delta)^s u)(x) = \partial^s_i ((-\Delta)^s u)(x).
\end{split}
\end{equation}
That is, the operators $(-\Delta)^s$ and $\nabla^s$ commute.
\end{Remark}

\subsection{A nonlocal functional based on the fractional gradient}

This subsection aims to show that most of the analysis developed for the nonlocal functional $J_{ACF}^s$ can be extended to the case $g_u=|\nabla^s u |^2$. More precisely, for $s \in (0,1)$ and $\varepsilon > 0$, we consider the following functional:
\begin{equation} \label{eq:ACFnonlocale-cs}
\begin{split}
    J_{ACF}^{s}(u,R) :&= \frac{n\omega_n}{R^{1+s}}\int_{0}^{R} r^s \int_{ B_r^c} K^s_r(0,y)|\nabla^s u(y)|^2\, dy \, dr \\
    &= \frac{n\omega_n}{R^{1+s}}\int_{0}^{R} r^s M_s(|\nabla^s u|^2,r)(0)\,dr,
    \end{split}
\end{equation}
where $u \in C^{s+\varepsilon}(\R^n) \cap L^1_{s/2}(\R^n)$, $R > 0$, $K^s_r$ is the Poisson kernel defined in~\eqref{poissondefn}, $|\nabla^s u|$ denotes the Euclidean norm of $\nabla^s u$ in $\R^n$, and $M_s$ is the $s$-mean operator defined in~\eqref{def-M}.

We begin by establishing the scaling properties of $J_{ACF}^s$.

\begin{Proposition}[Scaling invariance]\label{prop:scaling-cs}
Let $u \in C^{s+\varepsilon}(\R^n) \cap L^1_{s/2}(\R^n)$ and $R > 0$. Then, the functional $J_{ACF}^{s}$ is invariant under the scaling
\[
u_\lambda(x) := \frac{1}{\lambda^s} u(\lambda x),
\]
i.e.,
\[
J_{ACF}^{s}\left(u_\lambda, \frac{R}{\lambda}\right) = J_{ACF}^{s}(u, R).
\]
\end{Proposition}

\begin{proof} The result follows from the same argument used in the proof of Proposition~\ref{prop:scaling}, by taking into account the scaling behavior of the fractional gradient. Indeed, for $\lambda > 0$, it follows
\begin{align*}
\nabla^s u_\lambda(x) &= \mu_{n,s} \int_{\R^n} \frac{(u_\lambda(y) - u_\lambda(x))(y - x)}{|y - x|^{n+s+1}}\, dy \\
&= \mu_{n,s}\,\lambda^{-s} \int_{\R^n} \frac{(u(\lambda y) - u(\lambda x))(y - x)}{|y - x|^{n+s+1}}\, dy \\
&= \mu_{n,s}\int_{\R^n} \frac{(u(y) - u(\lambda x))(y - \lambda x)}{|y - \lambda x|^{n+s+1}}\, dy = \nabla^s u(\lambda x).
\end{align*}
\end{proof}

In view of the scaling property stated in Proposition \ref{prop:scaling-cs}, 
the proof of Theorem \ref{Thm:sACF}, in the case of $g_u=|\nabla^s u|^2$, follows applying the same argument. For this reason, we omit the details.

Furthermore, for sufficiently regular functions, the functional $J_{ACF}^s$ recovers the classical one-phase formula as $s \to 1^-$.

\begin{Theorem}\label{Thm:ACF-lim-cs}
Let $s_0 \in (0,1)$. Let {$u \in C^2_0(\R^n)$} and $|\nabla^s u |^2 \in C_{\textrm{loc}}^{2}(\R^n) \cap L_{s_0}^1(\R^n)$. Then, for $R>0$,
\[
J_{ACF}^{s}(u, R) \longrightarrow  \frac{1}{R^2} \int_{B_R} \frac{|\nabla u(y)|^2}{|y|^{n-2}} \, dy \quad \text{as } s \to 1^{-}.
\]
\end{Theorem}

\begin{proof}
Let $s\in [s_0,1)$. Arguing as in the proof of Theorem \ref{Thm:ACF-lim}, we have
\begin{equation*}
\begin{split}
    M_s(| \nabla^s u|^2, r) (0)&=\int_{B_r^c}K^s_r(0,y)|\nabla^su(y)|^2 \, dy\\
    &=  \int_{B_r^c}K^s_r(0,y) \left(|\nabla^su(y)|^2-|\nabla u(y)|^2 \right)\, dy+ \int_{B_r^c}K^s_r(0,y)|\nabla u(y)|^2 \, dy=I_1+I_2.
    \end{split}
\end{equation*}
By \cite[Proposition 4.4]{CS23}, we know that 
\begin{equation}\label{convunifCS}
\lim_{s\to 1^-}\|\nabla^s u-\nabla u\|_{L^\infty(\R^n)}=0.\
\end{equation}
Hence, we obtain
\begin{equation*}
    I_1\leq \sup_{\R^n}\left| |\nabla^s u|^2-|\nabla u|^2\right|\int_{\R^n}K^s_r(0,y)\, dy\longrightarrow 0
\end{equation*}
as $s\to1^-$. Thus, by \cite[Appendix C]{abatangelo2015large}, we conclude that 
\begin{equation*}
    \lim_{s\to 1^-}M_s(| \nabla^s u|^2, r) (0)=\lim_{s\to 1^-}\int_{B_r^c}K^s_r(0,y)|\nabla u(y)|^2 \, dy =\fint_{\partial B_r}|\nabla u(y)|^2 \, d\mathcal{H}^{n-1}(y).
\end{equation*}
Now, as in the proof of Theorem \ref{Thm:ACF-lim}, by integrating from $0$ to $R$ and applying the coarea formula, we obtain the desired claim.

\end{proof}

We also obtain the following estimate, analogous to Proposition~\ref{stima-tipo-locale}, whose proof follows in the same way.

\begin{Proposition}[ACF-type estimate]\label{stima-tipo-locale-cs}
Let $s \in (0,1)$ and $\varepsilon > 0$ be fixed. Assume $u \in C^{s+\varepsilon}_{\mathrm{loc}}(\R^n) \cap L^1_{s/2}(\R^n)$. Then, for all $R \in (0,1)$,
\[
J_{ACF}^{s}(u, R) \leq \frac{C}{R^{2s}} \int_{\R^n} \frac{|\nabla^s u(x)|^2}{|x|^{n-2s}} \, dx,
\]
where $C$ is a constant depending only on $n$ and $s$.
\end{Proposition}

\section{A nonlocal Bochner-type formula and some applications of the monotonicity formula} \label{section_Bochner}

The Bochner formula is one of the fundamental tools in geometric analysis. It provides a deep identity connecting the analytic behavior of smooth functions and differential forms with the underlying geometry of the manifold, through the Ricci tensor. Specifically, for a smooth function $f \in C^\infty(M)$ on a Riemannian manifold $(M,g)$, the formula reads 
\begin{equation}\label{eq:bochner-gen}
\frac{1}{2}\Delta |\nabla f|^2 =\|D^2 f\|^2 + \langle \nabla f, \nabla(\Delta f)\rangle + \mathrm{Ric}(\nabla f,\nabla f).
\end{equation}
Thus, the Bochner identity relates the Laplacian of the squared gradient norm to both the Hessian of $f$ and the Ricci curvature. In the special case of the Euclidean space $(\mathbb{R}^n, g_{\mathrm{eucl}})$, the Ricci tensor vanishes identically, and the identity \eqref{eq:bochner-gen} reduces to
\begin{equation}\label{eq:bochner-euclidean}
    \frac{1}{2} \Delta |\nabla f|^2 
    = \|D^2 f\|^2 + \langle \nabla f, \nabla (\Delta f) \rangle
\end{equation}
which is a key tool to study conditions ensuring the subharmonicity of $|\nabla f|^2$, see for example \cite{garofalo2023note} in the case of Carnot groups.  Thus, in order to analyze the condition $(-\Delta )^s G_{u} \leq 0$, which is required for the validity of the monotonicity formula in Theorem \ref{Thm:sACF}, we aim in this section to investigate a possible nonlocal analogue of \eqref{eq:bochner-euclidean}, involving the operator $(-\Delta)^s G_u$. This is the content of the following result.

\begin{Proposition}\label{lem:cc1} Let $s \in (0,1)$ and let $u \in C^{3}(\R^n) \cap L_{s}^2(\R^n)$ be such that $ G_u \in L_s^1(\R^n)$. Then the following identity holds
    \begin{equation} \label{eq:nonlocalboch}
    \begin{aligned}
        (-\Delta)^s G_u (x) &= 2 C_{n,s} \int_{\R^n} \frac{ (u(x)-u(x-z))(-\Delta )_x^s (u(x)-u(x-z))}{\nor{z}^{n+2s}} \, dz  \\ 
        &- C_{n,s}^2\int_{\R^n} \!\int_{\R^n} \frac{\left( u(x)- u(x-z) - u(y)  + u(y-z)\right)^2}{\nor{x-y}^{n+2s}\nor{z}^{n+2s}}  \, dy \, dz \quad \text{for }x \in \R^n
       \end{aligned}
    \end{equation}
    where the constant $C_{n,s}$ is defined in \eqref{def_c}.
\end{Proposition}

\begin{Remark} \label{rmk:s-sub_Gu}
    Notice that in case of $(-\Delta)^s u = c$ in $\R^n$, with $c$ constant, then   \eqref{eq:nonlocalboch} guarantees that $(-\Delta)^s G_u \leq 0$ in $\R^n$. In general, on a set $\Omega \subset \R^n$, the same conclusion follows if one can  show that
    \begin{equation*}
        \begin{aligned}
            ~ \qquad    \int_{\R^n}& \frac{ (u(x)-u(x-z))(-\Delta )_x^s (u(x)-u(x-z))}{\nor{z}^{n+2s}} \, dz \\  &\leq \frac{C_{n,s}}{2}\int_{\R^n} \!\int_{\R^n} \frac{\left( u(x)- u(x-z) - u(y)  + u(y-z)\right)^2}{\nor{x-y}^{n+2s}\nor{z}^{n+2s}}  \, dy \, dz \quad \text{for }x \in \Omega.
        \end{aligned}
    \end{equation*}
    In what follows we will focus on some examples for which this condition can be verified, see Section \ref{s:sub_Gu}.
\end{Remark}

We now provide the proof of Proposition \ref{lem:cc1}.

\begin{customproof}{Proposition \ref{lem:cc1}}
The delicate step in the proof of this result is to justify the interchange between the fractional Laplacian and the integral sign. Indeed, once this has been established, Lemma \ref{lemma:tec} immediately yields
  \begin{align*}
      (-\Delta)^s G_u &=  C_{n,s}\int_{\R^n} \frac{ (-\Delta)^s(u(x)-u(x-z))^2}{\nor{z}^{n+2s}} \, dz \\ 
     &=2C_{n,s} \int_{\R^n} \frac{ (u(x)-u(x-z))(-\Delta )_x^s (u(x)-u(x-z))}{\nor{z}^{n+2s}} \, dz  \\ 
       & \qquad - C_{n,s}^2 \int_{\R^n} \frac{1}{\nor{z}^{n+2s}} \int_{\R^n} \frac{\left( u(x)- u(x-z) - u(y)  + u(y-z)\right)^2}{\nor{x-y}^{n+2s}} \, dy \, dz  \quad \text{for } x  \in \R^n.
  \end{align*}
  Therefore, the crucial step in the proof is to rigorously justify the first passage. In particular, by employing the representation of the fractional Laplacian as a weighted second-order difference quotient, one needs to ensure the summability of the integral
 \begin{equation*}
      \int_{\R^n} \!\int_{\R^n} \frac{(u(x+k)-u(x+k-h))^2 + (u(x-k)-u(x-k-h))^2 -2 (u(x)-u(x-h))^2  }{\nor{h}^{n+2s} \nor{k}^{n+2s}} \, dh \, dk .
  \end{equation*}
In particular, the nontrivial singularity arises around the origin. Expanding $u$ twice near $x$ up to the third order terms, we focus on the least order terms that do not vanish. Notice that 
\begin{align*}
    u(x+k)-u(x+k-h) &= \ps{\nabla u (x)}{h} + \ps{D^2 u (x) h}{k} - \frac{1}{2}  \ps{D^2 u (x) h}{h} + g_1(h,k)  \\ 
    u(x-k)-u(x-k-h) &=  \ps{\nabla u (x)}{h} - \ps{D^2 u (x) h}{k} - \frac{1}{2}  \ps{D^2 u (x) h}{h} + g_2(h,k) \\ 
    u(x)-u(x-h) &= \ps{\nabla u (x)}{h} - \frac{1}{2} \ps{D^2 u (x) h}{h}  + g_3(h,k) \\
\end{align*}
The remaining term of lowest order can be estimated as
\begin{align*}
    \int_{B_\varepsilon(0)} \int_{B_\varepsilon(0)} \frac{ 2 \left(\ps{D^2 u (x) h}{k} \right)^2}{\nor{h}^{n+2s} \nor{k}^{n+2s}} \, dh \, dk 
    &\leq 2 \norm{D^2 u(x)}^2 \int_{B_1} \int_{B_1} \frac{ \nor{h}^2 \nor{k}^2}{\nor{h}^{n+2s} \nor{k}^{n+2s}} \, dh \, dk  \\
    &= 2 \norm{D^2 u(x)}^2 \omega_n^2 \int_{0}^{1} \int_{0}^{1} \frac{\rho^2 r^2 }{\rho^{n+2s} r^{n+2s}} \rho^{n-1} r^{n-1} \, d\rho \, dr \\  &= \frac{\norm{D^2 u(x)}^2 \omega_n^2}{2(1-s)^2},
\end{align*}
that allows one to conclude since $u \in C^3(\R^n).$
\end{customproof}

Now, we focus on the convergence of \eqref{eq:nonlocalboch} to \eqref{eq:bochner-euclidean} as $s \to 1^{-}$.

\begin{Proposition} \label{Prop.convBoch}
     Let $s_0 \in (0,1)$. Let $\Omega \subset \R^n$ and consider $u \in C^{4}(\R^n) \cap L_{s_0}^2(\R^n), G_u \in L_{s_0}^1(\R^n)$. For every $x \in \Omega,$ the formula \eqref{eq:nonlocalboch} tends to \eqref{eq:bochner-euclidean} as $s \to 1^{-}$.
\end{Proposition}

This  stability result can be achieved combining Lemma \ref{lemma:Asymp1} with \cite[Proposition 4.4]{di2012hitchhiker's} to have
    \begin{align} \label{eq:lim1}
       &(-\Delta)^s G_u \to   2 (-\Delta)( |\nabla u |^2) , \\ \label{eq:lim2}
       &2 C_{n,s} \int_{\R^n} \frac{ (u(x)-u(x-z))(-\Delta )_x^s (u(x)-u(x-z))}{\nor{z}^{n+2s}} \, dz \to 4\ps{\nabla u}{\nabla((-\Delta) u)} ,
    \end{align}
    as $s \to 1^{-}$.
Furthermore, the convergence of the second term on the right-hand side of \eqref{eq:nonlocalboch} is established in the following result.

\begin{Lemma} \label{lemma:hessnonloc}
     Let $s_0 \in (0,1)$ and $u \in C^{4}(\R^n) \cap L_{s_0}^2(\R^n)$ and $ G_u \in L_{s_0}^1(\R^n)$. 
     Then, for any $x \in \R^n$,
     \begin{equation}
         \lim_{s \to 1}   (C_{n,s})^2 \int_{\R^n} \!\int_{\R^n} \frac{\left( u(x)- u(x-z) - u(y)  + u(y-z)\right)^2}{\nor{x-y}^{n+2s}\nor{z}^{n+2s}}  \, dy \, dz  = 4\norm{D^2 u(x) }^2.
     \end{equation}
\end{Lemma}

\begin{proof}
 Let us consider any $s\in[s_0,1)$. We begin by splitting the integral
     \begin{align} \label{int1}
           &\int_{\R^n} \!\int_{\R^n} \frac{\left( u(x)- u(x-z) - u(y)  + u(y-z)\right)^2}{\nor{x-y}^{n+2s}\nor{z}^{n+2s}}  \, dy \, dz = \mathcal{I}_1 + \mathcal{I}_2 \\
           \intertext{where }
           \label{defI1}\mathcal{I}_1 &:= \int \!\int_{B_1(x) \times B_1} \frac{\left( u(x)- u(x-z) - u(y)  
           + u(y-z)\right)^2}{\nor{x-y}^{n+2s}\nor{z}^{n+2s}}  \, dy \, dz \\ \nonumber
          \mathcal{I}_2 &:= \int \!\int_{(B_1(x) \times B_1)^c}  \frac{\left( u(x)- u(x-z) - u(y)  + u(y-z)\right)^2}{\nor{x-y}^{n+2s}\nor{z}^{n+2s}}  \, dy \, dz.
       \end{align} 
       Let us first focus on $\mathcal{I}_2$. Since $(B_1(x) \times B_1)^c = \left( B_1^c(x) \times B_1 \right) \cup \left( \R^n \times B_1^c \right)$, we have 
    \begin{equation}\label{spezzo-I2}
      \begin{split}
  \mathcal{I}_2 = \int \!\int_{B_1^c(x) \times B_1} &\frac{\left( u(x)- u(x-z) - u(y) + u(y-z)\right)^2}{|x-y|^{n+2s}|z|^{n+2s}} \, dy \, dz \, \\
&+\int \!\int_{\R^n \times B_1^c} \frac{\left( u(x)- u(x-z) - u(y) + u(y-z)\right)^2}{|x-y|^{n+2s}|z|^{n+2s}} \, dy \, dz=:\mathcal{I}_{21}+\mathcal{I}_{22}.
  \end{split}
    \end{equation}
For the first term $\mathcal{I}_{21}$, we have
\begin{equation}\label{spezzoI21}
\begin{split}
\int&\int_{B_1^c(x) \times B_1} \frac{\left( u(x)- u(x-z) - u(y) + u(y-z)\right)^2}{|x-y|^{n+2s}|z|^{n+2s}} \, dy \, dz \\ &\leq  2\int \!\int_{B_1^c(x) \times B_1} \frac{\left( u(x)- u(x-z)\right)^2}{|x-y|^{n+2s}|z|^{n+2s}} \, dy \, dz + 2\int \!\int_{B_1^c(x) \times B_1} \frac{\left( u(y)- u(y-z)\right)^2}{|x-y|^{n+2s}|z|^{n+2s}} \, dy \, dz.
\end{split}
\end{equation}
For the first term in \eqref{spezzoI21}, we deduce
\begin{equation}\label{primopezzoI21}
\begin{split}
\int \!\int_{B_1^c(x) \times B_1} &\frac{(u(x)-u(x-z))^2}{|x-y|^{n+2s}|z|^{n+2s}} \, dy \, dz \\
&= \left( \int_{B_1^c(x)} \frac{dy}{|x-y|^{n+2s}} \right) \left( \int_{B_1} \frac{(u(x)-u(x-z))^2}{|z|^{n+2s}} \, dz \right)\\
&= n\omega_n \left(\int_1^\infty \rho^{n-1-(n+2s)} \, d\rho \right)\left( \int_{B_1} \frac{(u(x)-u(x-z))^2}{|z|^{n+2s}} \, dz \right)\\
&\leq \frac{n\omega_n}{2s} \|\nabla u\|_{L^\infty(B_1(x))}^2 \int_{B_1} |z|^{2-n-2s} \, dz \\
&= \|\nabla u\|_{L^\infty(B_1(x))}^2 \frac{n^2\omega_n^2}{4s(1-s)}
\end{split}
\end{equation}
where in the inequality, since $u \in C^4(\R^n)$, $\nabla u$ is locally bounded in $B_1(x)$, we use  $|u(x)-u(x-z)| \le |\nabla u|_{L^\infty(B_1(x))} |z|$ for every $z \in B_1$. 

Moreover, for the second term in \eqref{spezzoI21}, we have
\begin{equation}\label{secondopezzoI21}
    \begin{split}
         \int \!\int_{B_1^c(x) \times B_1} &\frac{(u(y)-u(y-z))^2}{|x-y|^{n+2s}|z|^{n+2s}} \, dy \, dz \\
         &= \int_{B_1^c(x)} \frac{1}{|x-y|^{n+2s}} \left( \int_{B_1} \frac{(u(y)-u(y-z))^2}{|z|^{n+2s}} \, dz \right) dy\\
         &\leq \int_{B_1^c(x)} \frac{1}{|x-y|^{n+2s}} \left( \int_{\R^n} \frac{(u(y)-u(y-z))^2}{|z|^{n+2s}} \, dz \right) dy\\
         &\leq \frac{1}{C_{n,s}}\int_{B_1^c(x)} \frac{G_u(y)}{|x-y|^{n+2s}} dy\leq \frac{1}{C_{n,s}}\int_{B_1^c(x)} \frac{G_u(y)}{|x-y|^{n+2s_0}} dy \leq \frac{1}{C_{n,s}} \| G_u\|_{L_{s_0}^1(\R^n)} .    
    \end{split}
\end{equation}

On the other hand, for $\mathcal{I}_{22}$, we have 
\begin{equation}\label{spezzoI22}
\begin{split}
\int&\!\!\int_{\R^n \times B^c_1} \frac{\left( u(x)- u(x-z) - u(y) + u(y-z)\right)^2}{|x-y|^{n+2s}|z|^{n+2s}} \, dy \, dz \\ &\leq  2\int\!\!\int_{\R^n \times B^c_1} \frac{\left( u(x)- u(y)\right)^2}{|x-y|^{n+2s}|z|^{n+2s}} \, dy \, dz + 2\int\!\!\int_{\R^n \times B_1^c} \frac{\left( u(y-z)- u(x-z)\right)^2}{|x-y|^{n+2s}|z|^{n+2s}} \, dy \, dz.
\end{split}
\end{equation}
For the first term in \eqref{spezzoI22}, one gets
\begin{equation}\label{primopezzoI22}
\begin{split}
\int\!\!\int_{\R^n \times B_1^c } \frac{(u(x)-u(y))^2}{|x-y|^{n+2s}|z|^{n+2s}} \, dy \, dz &= \left( \int_{B_1^c} \frac{dz}{|z|^{n+2s}} \right) \left( \int_{\R^n} \frac{(u(x)-u(y))^2}{|x-y|^{n+2s}} \, dy \right)\\
&=  \frac{n\omega_n}{2s\,C_{n,s}}G_u(x) .
\end{split}
\end{equation}
while for the second term in \eqref{spezzoI22}, we have
\begin{equation}\label{secondopezzoI22}
    \begin{split}
         \int \!\!\int_{\R^n\times B_1^c } \frac{(u(x-z)-u(y-z))^2}{|x-y|^{n+2s}|z|^{n+2s}} \, dy \, dz 
         &= \int_{B_1^c} \frac{1}{|z|^{n+2s}} \left( \int_{\R^n} \frac{(u(x-z)-u(y-z))^2}{|x-y|^{n+2s}} \, dy \right) dz\\
         &\leq \frac{1}{C_{n,s}}\int_{B_1^c} \frac{G_u(x-z)}{|z|^{n+2s}} dz\\
         &=\frac{1}{C_{n,s}}\int_{B_1^c(x)} \frac{G_u(\xi)}{|x-\xi|^{n+2s}} d\xi \leq        \frac{1}{C_{n,s}} \| G_u\|_{L_{s_0}^1(\R^n)}.
    \end{split}
\end{equation}

Thus, taking into account that $C_{n,s}\sim 1-s$ as $s\to 1^-$ (see, e.g., \cite[Corollary 4.2]{di2012hitchhiker's}) and combining the estimates in \eqref{spezzo-I2}, \eqref{spezzoI21}, \eqref{primopezzoI21}, \eqref{secondopezzoI21}, \eqref{spezzoI22}, \eqref{primopezzoI22} and \eqref{secondopezzoI22}, we find that, as $s\to 1^-$ 
\begin{equation*}
    \begin{split}
    (C_{n,s})^2 \mathcal{I}_2 &\leq (C_{n,s})^2 \mathcal{I}_{21}+(C_{n,s})^2 \mathcal{I}_{22} \\&\le \frac{C_{n,s}^2 \,n^2\,\omega^2_n}{2s(1-s)} \|\nabla u\|_{L^\infty(B_1(x))}+ \frac{C_{n,s}\,n\,\omega_n}{s}G_u(x) + 4C_{n,s} \| G_u\|_{L_{s_0}^1(\R^n)} \longrightarrow0
    \end{split}
\end{equation*}
which proves that $C^2_{n,s}\mathcal{I}_2\to 0$, as $s\to 1^-$.

Now, let us consider $\mathcal{I}_1$ in \eqref{defI1}. Changing in \eqref{int1} $z$ with $-z$ and letting $y=x+h$ with $h\in B_1$, for some $\eta,\zeta \in \R^n$, expanding with an integral remainder we obtain
       \begin{align*}
           ( u(x)-u(x+&z) - u(y)  + u(y+z))^2= \\
           &=[u(x+h+z)-u(x+z) - (u(x+h)-u(x)) ]^2\\
&= \left[\nabla u (x+z)\cdot h + \frac{1}{2}D^2u(x+z) h \cdot h +  \sum_{\substack{\alpha \in \mathbb{N}^n, \\ |\alpha|=3}}  \frac{3}{\alpha!} \int_0^1 (1-t)^2 \p^\alpha u(x+z+th) \, dt \, h^{\alpha} \right. \\
& \qquad\ \left. - \nabla u (x)\cdot h + \frac{1}{2}D^2u(x) h \cdot h +  \sum_{\substack{\alpha \in \mathbb{N}^n, \\ |\alpha|=3}}  \frac{3}{\alpha!} \int_0^1 (1-t)^2 \p^\alpha u(x+th) \, dt \, h^{\alpha}
\right]^2
 \\
 &= \left[ D^2 u (x) z\cdot h + \frac{1}{2} D^{3}u (\eta) z z \cdot h + \frac{1}{2 } D^3 u (\zeta) z h \cdot h
\right. \\ & \qquad\ \left.+ \sum_{\substack{\alpha \in \mathbb{N}^n, \\ |\alpha|=3}}  \frac{3}{\alpha!} \int_0^1 (1-t)^2 \left( \p^\alpha u(x+z+th) - \p^\alpha u(x+th)\right) \, dt \, h^{\alpha}
\right]^2 \\
&=\left[ D^2 u (x) z\cdot h + \frac{1}{2} D^{3}u (\eta) z z \cdot h + \frac{1}{2 } D^3 u (\zeta) z h \cdot h \right. \\ &\qquad \ \left.
+ \sum_{\substack{\alpha \in \mathbb{N}^n, \\ |\alpha|=3}}  \frac{3}{\alpha!} \int_{0}^{1} \int_0^1 (1-t)^2 \ps{\nabla(\p^\alpha u)(x+sz+ th)}{z} \, ds \, dt \, h^{\alpha}
\right]^2
\end{align*}
       where we use the notation $D^3 u(p) v w \cdot q   := \sum_{i,j,k=1}^n\p^3_{ijk} u(p) v_i w_j q_k$ for $p,v,w,q \in \R^n$.
    When integrating the above squared expression, all mixed terms containing odd powers vanish by symmetry.

Let us isolate the leading term, which comes from the square of the bilinear form $D^2 u(x) z \cdot h$. Expanding it, we get
\begin{align*}
   \int_{B_1} \int_{B_1} &\frac{\left( D^2 u(x) z \cdot h\right)^2}{|h|^{n+2s} |z|^{n+2s}}\, dh \, dz \\
   &= \sum_{i,j,k,l=1}^n \p_{ij}^2 u(x)\p_{kl}^2 u(x) \left( \int_{B_1} \frac{z_i z_k}{|z|^{n+2s}}\, dz \right) \left( \int_{B_1} \frac{h_j h_l}{|h|^{n+2s}}\, dh \right).
\end{align*}
By symmetry, $\int_{B_1} y_i y_k |y|^{-n-2s} dy = 0$ whenever $i \neq k$. For $i=k$, passing to polar coordinates yields
$$ \int_{B_1} \frac{y_i^2}{|y|^{n+2s}}\, dy = \frac{1}{n} \int_{B_1} \frac{|y|^2}{|y|^{n+2s}}\, dy = \frac{\omega_n}{n} \int_0^1 \rho^{1-2s}\, d\rho = \frac{\omega_n}{2(1-s)}. $$
Therefore, the multiple sum collapses entirely to the terms where $i=k$ and $j=l$
\begin{equation*} 
   \int_{B_1} \int_{B_1} \frac{\left( D^2 u(x) z \cdot h\right)^2}{|h|^{n+2s} |z|^{n+2s}}\, dh \, dz = \sum_{i,j=1}^n (\p_{ij}^2 u(x))^2 \left( \frac{\omega_n}{2(1-s)} \right)^2 = \frac{\omega_n^2}{4(1-s)^2} \|D^2u(x)\|^2.
\end{equation*}

As for the remaining terms in the squared expansion, they involve integrals of higher-order polynomials. Using polar coordinates, it is easy to check that an integral of the form $\int_{B_1} y^\alpha |y|^{-n-2s} dy$ with $|\alpha| \ge 4$ produces a singularity of order at most $O((2-s)^{-1})$. Cross-products with $|\alpha| = 2$ yield $O((1-s)^{-1})$. Consequently, all the remainder terms behave at worst as $O((1-s)^{-1})$ and vanish when multiplied by the scaling factor $(1-s)^2$.

From \cite[Corollary 4.2]{di2012hitchhiker's} it is known that the asymptotic behavior of the constant $C_{n,s}$ as $s \to 1^{-}$ is 
\begin{equation}\label{limitCns}
\lim_{s\to 1^{-}} \frac{C_{n,s}}{1-s} = \frac{4}{\omega_n}.
\end{equation}
Thus, gathering the previous facts, we conclude
\begin{align*}
    \lim_{s\to 1^{-}} C_{n,s}^2 &\int_{\R^n} \!\int_{\R^n} \frac{\left( u(x)- u(x-z) - u(y)  + u(y-z)\right)^2}{\nor{x-y}^{n+2s}\nor{z}^{n+2s}}  \, dy \, dz \\ 
    &= \lim_{s\to 1^{-}} \left( \frac{C_{n,s}}{1-s} \right)^2 (1-s)^2 \left[ \frac{\omega_n^2}{4(1-s)^2} \|D^2u (x)\|^2 + O\left(\frac{1}{1-s}\right) \right] \\
    &= \frac{16}{\omega_n^2} \frac{\omega_n^2}{4} \|D^2u(x)\|^2 = 4 \|D^2u (x)\|^2. 
\end{align*}

\end{proof}

\begin{customproof}{Proposition \ref{Prop.convBoch}}
    The result follows directly combining \eqref{eq:lim1}, \eqref{eq:lim2} and Lemma \ref{lemma:hessnonloc}.
\end{customproof}

\subsection{On the $s$-subharmonicity of $G_u$} \label{s:sub_Gu}

The compatibility condition $(-\Delta)^s G_u \le 0$, which stands at the core of Theorem \ref{Thm:sACF}, can be systematically verified for some classes of equations. By exploiting the nonlocal Bochner identity established in Proposition \ref{lem:cc1}, as an example, we show that if the semilinear term has a nonincreasing behavior one can ensure the $s$-subharmonicity of $G_u$. As a simple case, one should note that, whenever $(-\Delta)^s u$ is constant in $\R^n$, Proposition \ref{lem:cc1} immediately yields the compatibility condition $(-\Delta)^s G_u \leq 0$ in $\R^n$.

\begin{Corollary}\label{cor:semilinear}
Let $s\in(0,1)$ and let $u \in C^3(\mathbb{R}^n) \cap L^2_s(\mathbb{R}^n)$ be a solution to the semilinear equation
\begin{equation}\label{eq:semilinear_generic}
(-\Delta)^s u(x) = f(u(x)) \quad \text{in } \mathbb{R}^n,
\end{equation}
where $f \in \text{Lip}_{\text{loc}}(\R)$ is monotone nonincreasing. Assuming that $G_u \in L^1_s(\mathbb{R}^n)$, then 
$$
(-\Delta)^s G_u(x) \le 0 \quad \text{for all } x \in \mathbb{R}^n.
$$
\end{Corollary}

\begin{proof}
We invoke the identity \eqref{eq:nonlocalboch} established in Proposition \ref{lem:cc1}. Since the second integral in \eqref{eq:nonlocalboch} is clearly nonpositive, it suffices to show that the first term is also nonpositive.

Due to the linearity of the fractional Laplacian and the fact that $u$ solves \eqref{eq:semilinear_generic}, the first integral on the right-hand side can be rewritten as
\[
2C_{n,s}  \int_{\R^n} \frac{(u(x)-u(x-z))(f(u(x)) - f(u(x-z)))}{|z|^{n+2s}} \, dz.
\]
Because $f$ is monotone nonincreasing, the inequality $(a - b)(f(a) - f(b)) \le 0$ holds for all real numbers $a, b \in \mathbb{R}$. Thus, we infer that the integrand is pointwise nonpositive for every $x, z \in \mathbb{R}^n$. This immediately yields $(-\Delta)^s G_u(x) \le 0$, concluding the proof.
\end{proof}

\subsection{A Liouville-type result} 

Since the pioneering works of Keller \cite{Keller} and Osserman \cite{Osserman}, Liouville-type results for semilinear elliptic equations have played a central role in the qualitative theory of PDEs. In the nonlocal setting, such results are usually derived through maximum principles, moving planes techniques, or the analysis of equivalent integral equations, see for instance \cite{ChenLiOu2006,ChenFangYang2014}. On the other hand, Liouville properties for fractional harmonic functions have been extensively investigated \cite{ChenDAmbrosioLi2015,Fall2016}.
The approach developed here is different. By the fractional Bochner identity \eqref{eq:nonlocalboch}, the nonlocal energy density $G_u$ is $s$-subharmonic whenever $f$ is monotone nonincreasing, as shown in Corollary \ref{cor:semilinear}. The argument combines the mean-value property  with the estimate for the ACF-type quantity in Proposition \ref{stima-tipo-locale}.

\begin{Theorem}\label{thm:liouville_unified}
Let $s\in(0,1)$ and let $u \in C^3(\mathbb{R}^n) \cap L^2_s(\mathbb{R}^n)$ be a solution to the semilinear equation \eqref{eq:semilinear_generic} where $f \in \text{Lip}_{\text{loc}}(\R)$ is monotone nonincreasing and $f(0)=0$. If $G_u \in L^1_s(\mathbb{R}^n)$, then $u$ must be a constant $c \in \mathbb{R}$ such that $f(c)=0$.
\end{Theorem}

\begin{proof}
In light of Corollary \ref{cor:semilinear}, for every $R\geq 0$, $(-\Delta)^s G_u\leq 0$ in $B_R$. Thus, by the mean value property of $G_u$, see Proposition \ref{prop-mean-1}, we obtain
\begin{equation*}
\begin{split}
     G_u(0) &= \frac{1+s}{R^{1+s}}\int_{0}^R \rho^s G_u(0) \, d\rho  \leq   C J_{ACF}^s(u,R) 
     \end{split}
     \end{equation*}
which, together with Proposition \ref{stima-tipo-locale}, gives 
\begin{equation}\label{stima-1-tuttiR}
G_u(0)\leq \frac{C}{R^{2s}}\int_{\R^n}\frac{G_u(x)}{|x|^{n-2s}}
\end{equation}
for all $R>0$. Moreover, since $u$ is a solution of \eqref{eq:semilinear_generic}, 
by Lemma \ref{lemma:tec}, we have that 
\begin{equation} \label{alg_id}
    G_u (x)=   2u(x) f(u(x))  -(-\Delta)^s u^2(x)   \quad \text{for }x \in \R^n.
\end{equation}
Observe that the conditions that $f$ is nonincreasing and $f(0)=0$ imply that 
\[
tf(t) = t \int_0^t f'(\tau) d\tau \leq 0 \quad \text{ for all }t \in \R. 
\]
Hence \eqref{alg_id} reduces to $G_u \leq -(-\Delta)^s u^2$.
Thus by the same argument employed in the proof of Lemma \ref{Lm:IntGradientEst_in0}, we obtain that 
\begin{equation}
    \int_{\R^n}\frac{G_u(x)}{|x|^{n-2s}}\,dx\leq\int_{\R^n}\frac{-(-\Delta)^su^2(x)}{|x|^{n-2s}}\,dx \leq C u(0)^2
\end{equation}
which, together with \eqref{stima-1-tuttiR}, ensures that 
\begin{equation}\label{last-tuttiR-2}
    G_u(0)\leq \frac{C_0}{R^{2s}}u^2(0)
\end{equation}
for any $R>0$ and some constant $C_0>0$ depending only on $n$ and $s$.

Letting $R\to \infty$ in \eqref{last-tuttiR-2}, we obtain $G_u(0)=0$. Hence by the definition of $G_u$, $u=u(0)$ almost everywhere in $\R^n$, and therefore everywhere by continuity. Thus $u\equiv c$, and the equation yields $f(c)=0$.

\end{proof}

\subsection{On the $s$-subharmonicity of $| \nabla^s u|^2$}
By adopting the alternative formulation with $| \nabla^s u|^2$ introduced in Section \ref{alternativedef}, and in view of Remark \ref{scambio}, one can verify the compatibility condition of the  monotonicity formula for some class of functions.

As a preliminary step, we recall the companion Bochner-type identity for $ |\nabla^s u|^2$.

\begin{Proposition}[Bochner-type identity for $\nabla^s$]\label{prop:bochner-cs}
Fix $s \in (0,1)$ and $\varepsilon > 0$. Let $u \in C^{s+\varepsilon}_{\mathrm{loc}}(\R^n) \cap L^1_{s/2}(\R^n)$ be such that $
|\nabla^s u|^2 \in C^{2s+\delta}_{\mathrm{loc}}(\R^n) \cap L^1_s(\R^n)$. Then the following identity holds
\begin{equation}\label{bochner-cs}
  (-\Delta)^s |\nabla^s u(x)|^2 = 2\left\langle \nabla^s u(x),\, \mathbf{(-\Delta)^s}(\nabla^s u(x)) \right\rangle - \sum_{i=1}^n G_{\partial^s_i u}(x),\quad x\in\R^n
\end{equation}
where $\langle \cdot,\cdot \rangle$ denotes the inner product in $\R^n$, 
\begin{equation}\label{partial-s}
\partial^{s}_i u(x) := \mu_{n, s} \int_{\R^n} \frac{(y_i - x_i)(u(y) - u(x))}{|y - x|^{n + s + 1}} \, dy, \quad i = 1, \dots, n,
\end{equation}
where $\mu_{n,s}$ is the constant in \eqref{mu-cs} and
\[
\mathbf{(-\Delta)^s}(\nabla^s u(x)) := \left((-\Delta)^s \partial^s_1 u(x), \dots, (-\Delta)^s \partial^s_n u(x)\right) \in \R^n.
\]
\end{Proposition}

\begin{Remark} \label{rmk:s-sub_nabla_s}
    Notice that if $u \in C_0^\infty(\R^n)$ and $(-\Delta)^s u = c$ in $\R^n$, with $c$ constant, then   \eqref{bochner-cs} guarantees that $(-\Delta)^s |\nabla^s u|^2 \leq 0$ in $\R^n$, since the operators $(-\Delta)^s$ and $\nabla^s$ commute, see Remark \ref{scambio}.  However, unlike in Remark \ref{rmk:s-sub_Gu}, the conditions on $u$ imply that $u\equiv 0$.

\end{Remark}

We now provide the proof of Proposition \ref{prop:bochner-cs}.

\begin{customproof}{Proposition \ref{prop:bochner-cs}}
By the linearity of $(-\Delta)^s$ and applying Lemma~\ref{lemma:tec}, we obtain
\begin{equation*}
\begin{split}
(-\Delta)^s|\nabla^s u(x)|^2 
&= (-\Delta)^s\left( \sum_{i=1}^n \left( \partial^s_i u(x) \right)^2 \right) 
= \sum_{i=1}^n (-\Delta)^s\left( \left( \partial^s_i u(x) \right)^2 \right) \\
&= \sum_{i=1}^n \left( 2 \partial^s_i u(x) \, (-\Delta)^s \partial^s_i u(x) - G_{\partial^s_i u}(x) \right),
\end{split}
\end{equation*}
which proves \eqref{bochner-cs}.
\end{customproof}
Also in this case, we recover the usual asymptotic behavior as $s \to 1^-$, yielding the parallel local formula.
\begin{Proposition}
For {$u \in C^3_0(\R^n)$}, the formula \eqref{bochner-cs} tends to \eqref{eq:bochner-euclidean} as $s \to 1^-$.
\end{Proposition}

\begin{proof}
By \cite[Proposition 4.4]{CS23}, we have that $\nabla^s u \to \nabla u$ uniformly on $\R^n$ as $s \to 1^-$. Hence, by \cite[Corollary 4.2]{di2012hitchhiker's}, we obtain
\[
(-\Delta)^s |\nabla^s u(x)|^2 \to (-\Delta) |\nabla u(x)|^2 \quad \text{and} \quad (-\Delta)^s(\partial^s_i u(x)) \to (-\Delta) (\partial_i u(x)),
\]
for $x \in \R^n$ and $i = 1, \ldots, n$, as $s \to 1^-$. Moreover, by Lemma \ref{lemma:Asymp1}, we have
\[
\sum_{i=1}^n G_{\partial^s_i u}(x) \to 2 \sum_{i=1}^n |\nabla (\partial_i u(x))|^2 = 2\|D^2 u(x)\|^2,\qquad x\in\R^n,
\]
as $s \to 1^-$, concluding the proof.
\end{proof}

As already discussed at the beginning of the section, the investigation of the  $s$-subharmonicity condition for $|\nabla^s u|^2$ relies on the use of the Bochner-type identity established in Proposition \ref{prop:bochner-cs}. In particular, by exploiting the commutativity property recalled in Remark \ref{scambio}, we are now in a position to state the following result.

\begin{Corollary}
     \label{coroll-ACF-csv2}
Let $u \in C^\infty_0(\R^n)$ be such that $(-\Delta)^s u =f$ in $\R^n$ and $g_u=|\nabla^s u|^2$ where
\[
f \in C^{s+\varepsilon}_{\mathrm{loc}}(\R^n) \cap L^1_{s/2}(\R^n), \text{ with } \ps{\nabla^s u}{\nabla^s f} \leq 0 \text{ in a neighborhood of the origin}.
\] 
Then the map
\begin{equation}
R \mapsto J_{ACF}^{s}(u, R) 
:= \frac{n \omega_n }{R^{1+s}} \int_0^R r^s \int_{ B_r^c} 
K^s_r(0,y)\, |\nabla^s u(y)|^2 \, dy \, dr
\end{equation}
is monotone increasing for $R > 0$ sufficiently small.
\end{Corollary}

\begin{proof}
Suppose that $(-\Delta)^s u = f$ in $\R^n$. 
From the assumptions on $f$ we know that there exists $\delta>0$ small so that
\[
\ps{\nabla^s u(x)}{\nabla^s f(x)} \leq 0 \quad \text{ for all }x \in B_{\delta}.
\]
Since $u \in C^\infty_0(\R^n)$, recalling \eqref{long-four} and \eqref{bochner-cs}, for any $x \in B_{\delta}$ we obtain
\begin{equation*}
\begin{split}
(-\Delta)^s |\nabla^s u(x)|^2 
&= 2 \sum_{i=1}^n \partial^s_i u(x)\, \partial^s_i \big((-\Delta)^s u\big)(x)
     - \sum_{i=1}^n G_{\partial^s_i u}(x) \\
&= 2 \ps{\nabla^s u(x)}{\nabla^s f(x)} 
     - \sum_{i=1}^n G_{\partial^s_i u}(x) \leq - \sum_{i=1}^n G_{\partial^s_i u}(x) \;\leq\; 0.
\end{split}
\end{equation*}
Therefore, the claim follows directly from Theorem \ref{Thm:sACF}.
\end{proof}

\appendix
\section{Asymptotics as $s \to 1^-$} \label{appendix}

We report here for completeness a fairly intuitive result about the uniform convergence of the   \textit{inner} product in $W^{s,2}$ to the Euclidean scalar product of the gradients. For the related energies, we found it contained in \cite[Lemma 3.3]{cinti2024fractional}.

\begin{Lemma} \label{lemma:Asymp1}
Let $s_0\in(0,1)$. For any $u,v \in C_{loc}^{1,\gamma}(\R^n) \cap L^2_{s_0}(\R^n)$ with $\gamma > 0$ and any compact set $K \subset \R^n$, it holds that 
\begin{equation}\label{eq:convgrad_compact}
\lim_{s \to 1^-} \sup_{x \in K} \left| \frac{C_{n,s}}{2} \int_{\R^n} \frac{(u(x)-u(y))(v(x)-v(y))}{|x-y|^{n+2s}} d  y - \ps{\nabla u(x)}{\nabla v(x)} \right| = 0.
\end{equation}
In particular, for $u=v$, we obtain
\begin{equation}\label{eq:convGu_compact}
\lim_{s \to 1^-} \sup_{x \in K} \left| \frac{G_u(x)}{2} - |\nabla u(x)|^2 \right| = 0.
\end{equation}
\end{Lemma}

\begin{proof}
Let $K$ be a compact set of $\R^n$ and let $\rho>0$ such that $K+B_1 \subset B_\rho$. Fix $s\in [s_0,1)$ and $x \in K$. We begin by splitting the integral
     \begin{align*}
           \int_{\R^n} &\frac{(u(x)-u(y))(v(x)-v(y))}{|x-y|^{n+2s}}\, dy \\ &= \int_{B_1} \frac{(u(x)-u(x+h))(v(x)-v(x+h))}{|h|^{n+2s}} \, dh  \\
           &\qquad +  \int_{B_1^c} \frac{(u(x)-u(x+h))(v(x)-v(x+h))}{|h|^{n+2s}} \, dh =: \mathcal{I}_1(x) + \mathcal{I}_2(x).
       \end{align*}

We first show that 
\begin{equation} \label{eq:stima_I2_tail}
\mathcal{I}_2(x) \leq C( \| u \|_{L^\infty(K)}^2 + \|v \|_{L^\infty(K)}^2 + \|u\|_{L_{s_0}^2(\R^n)}^2 + \|v\|_{L_{s_0}^2(\R^n)}^2)
\end{equation}
where $C>0$ is a constant depending on $n$, $s_0$ and $\rho$. Since $C_{n,s}\to 0 $ as $s \to 1^{-}$ this implies that  
       \begin{equation}\label{eq:step2_lim}
\lim_{s \to 1^-} \frac{C_{n,s}}{2} \sup_{x \in K} |\mathcal{I}_2(x)| = 0.
\end{equation}

In order to show \eqref{eq:stima_I2_tail}, we make use of Young's inequality, and for any $s_0\leq s$, we obtain
\begin{align*}
|\mathcal{I}_2(x)| &\le \int_{B_1^c} \frac{|u(x)-u(x+h)| \, |v(x)-v(x+h)|}{|h|^{n+2s_0}} d  h \\
&\le  \int_{B_1^c} \frac{|u(x)|^2 + |u(x+h)|^2 + |v(x)|^2 + |v(x+h)|^2}{|h|^{n+2s_0}} d  h \\
&\leq \left( \|u\|_{L^\infty(K)}^2 + \|v\|_{L^\infty(K)}^2 \right)\int_{B_1^c} \frac{1 }{|h|^{n+2s_0}} d  h + \int_{ B_1^c} \frac{ |u(x+h)|^2 + |v(x+h)|^2}{|h|^{n+2s_0}} d  h .
\end{align*}
In the second term, we let $y=x+h$ and observe that $|y|\leq |y-x|+\rho \leq (\rho+1)|y-x|$ since $x \in K \subset B_\rho$ and $h \in B_1^c$. In this way we obtain
\begin{equation*}
\begin{aligned}
    \int_{ B_1^c} \frac{ |u(x+h)|^2 + |v(x+h)|^2}{|h|^{n+2s_0}} d  h &=\int_{B_1^c(x)} \frac{ |u(y)|^2 + |v(y)|^2}{|y-x|^{n+2s_0}} d y \\ &\leq (2+\rho)^{n+2s_0} \left(\|u \|_{L_{s_0}^2(\R^n)}^2 + \|v \|_{L_{s_0}^2(\R^n)}^2 \right),
\end{aligned}
\end{equation*}
showing \eqref{eq:stima_I2_tail}.

We now focus on $\mathcal{I}_1(x)$.  Since $u,v \in C_{loc}^{1,\gamma}(\R^n)$, we can expand them around $x \in K$ as 
    \begin{align*}
        u(x+h) - u(x) &= \ps{\nabla u(x)}{h} + \mathcal{R}_u(x,h) \\
        v(x+h) - v(x) &= \ps{\nabla v(x)}{h} + \mathcal{R}_v(x,h)
    \end{align*}
    where the remainders satisfy $|\mathcal{R}_u(x,h)| \leq [\nabla u]_{C^\gamma(B_\rho)} |h|^{1+\gamma}$ and $|\mathcal{R}_v(x,h)| \leq [\nabla v]_{C^\gamma(B_\rho)} |h|^{1+\gamma}$ for $h \in B_1$.

    Multiplying the two increments, we obtain
    \begin{align*}
        (u(x)-u(x+h))&(v(x)-v(x+h)) = \ps{\nabla u(x)}{h}\ps{\nabla v(x)}{h} \\
        &\quad + \ps{\nabla u(x)}{h}\mathcal{R}_v(x,h) + \ps{\nabla v(x)}{h}\mathcal{R}_u(x,h) + \mathcal{R}_u(x,h)\mathcal{R}_v(x,h).
    \end{align*}
    With this, we further split  $\mathcal{I}_1(x)$ as 
    \begin{equation*}
          \begin{aligned}
        \mathcal{I}_1(x)=\mathcal{I}_1^M(x) &+ \mathcal{I}_1^E(x):=\int_{B_1} \frac{\ps{\nabla u(x)}{h}\ps{\nabla v(x)}{h}}{|h|^{n+2s}} \, dh
        \\
        &+ \int_{B_1} \frac{\ps{\nabla u(x)}{h}\mathcal{R}_v(x,h) + \ps{\nabla v(x)}{h}\mathcal{R}_u(x,h) + \mathcal{R}_u(x,h)\mathcal{R}_v(x,h)}{|h|^{n+2s}} \, dh.
    \end{aligned}
    \end{equation*}

    By symmetry of $B_1$, the main term $\mathcal{I}_1^M(x)$ can be explicitly computed and gives
    \begin{align*} 
        \mathcal{I}_{1}^M(x) &= \sum_{i,j=1}^n \p_{x_i} u(x)\p_{x_j} v(x) \int_{B_1} \frac{h_i h_j}{|h|^{n+2s}} \, dh \\
        &= \sum_{i=1}^n  \p_{x_i} u(x)\p_{x_i} v(x) \int_{B_1} \frac{h_i^2}{|h|^{n+2s}} \, dh \\
        &= \ps{\nabla u(x)}{\nabla v(x)} \frac{1}{n} \int_{B_1} |h|^{2-n-2s} \, dh = \frac{\omega_n}{2(1-s)}\ps{\nabla u(x)}{\nabla v(x)}.
    \end{align*}

    Recalling from \cite[Corollary 4.2]{di2012hitchhiker's} that $\lim_{s \to 1^-} \frac{C_{n,s}}{1-s} = \frac{4}{\omega_n}$, we get
\begin{equation} \label{eq:I_1M}
 \sup_{x \in K} \left| \frac{C_{n,s}}{2} \mathcal{I}_{1}^M(x)  - \ps{\nabla u(x)}{\nabla v(x)}\right| \leq \|\nabla u \|_{L^\infty(K)} \|\nabla v \|_{L^\infty(K)} \left| \frac{C_{n,s} \omega_n}{4(1-s)} - 1 \right| \to 0 \quad \text{as }s\to 1^-   
\end{equation}

    Now we estimate the error term $\mathcal{I}_1^E(x)$. By the H\"older bounds on the remainders, we obtain
    \begin{equation} \label{eq:I_1E}
    \begin{aligned}
              \sup_{x \in K} |\mathcal{I}_1^E(x)| &\leq C \|\nabla u \|_{C^\gamma(B_\rho)}  \|\nabla v \|_{C^\gamma(B_\rho)} \int_{B_1} \frac{|h|^{2+\gamma}}{|h|^{n+2s}} \, dh \\ &= C \omega_n \|\nabla u \|_{C^\gamma(B_\rho)}  \|\nabla v \|_{C^\gamma(B_\rho)}\int_{0}^1 r^{1+\gamma-2s} \, dr= \frac{C\|\nabla u \|_{C^\gamma(B_\rho)}  \|\nabla v \|_{C^\gamma(B_\rho)}}{2+\gamma-2s},
    \end{aligned}
    \end{equation}
 where $C$ is a constant depending only on $n$, whose value may change from line to line.

  Finally, we get
\begin{equation}\label{eq:triangle_main}
\begin{aligned}
\sup_{x \in K} &\left| \frac{C_{n,s}}{2} \int_{\R^n} \frac{(u(x)-u(y))(v(x)-v(y))}{|x-y|^{n+2s}}d  y - \ps{\nabla u(x)}{\nabla v(x)} \right| \\
&\le \sup_{x \in K} \left| \frac{C_{n,s}}{2} \mathcal{I}_{1}^M(x)  - \ps{\nabla u(x)}{\nabla v(x)}\right| + \frac{C_{n,s}}{2} \sup_{x \in K} |\mathcal{I}_1^E(x)| + \frac{C_{n,s}}{2} \sup_{x \in K} |\mathcal{I}_2(x)|,
\end{aligned}
\end{equation}
Collecting the estimates \eqref{eq:I_1M}, \eqref{eq:I_1E}, and \eqref{eq:stima_I2_tail}, and letting $s \to 1^-$, we obtain \eqref{eq:convgrad_compact}.

\end{proof}

\section*{Declarations}

\noindent{\bf Author contributions:}

all authors shared equally the writing and the reviewing of the manuscript.

\noindent{\bf Ethical Approval:} 

not applicable.

\noindent{\bf Competing interests:}

the authors declare that there is no conflict of interest. 

\noindent{\bf Funding:}

the authors are members of the Gruppo Nazionale per l’Analisi Matematica, la Probabilità e le loro Applicazioni (GNAMPA) of the Istituto Nazionale di Alta Matematica (INdAM) - CUP E5324001950001. F.F. and D.G. are partially supported by INDAM-GNAMPA project 2025 \say{At The Edge Of Ellipticity} and by PRIN project 20227HX33Z \say{Pattern formation in nonlinear phenomena} - CUP J53D2300361000. E.M.M. is partially supported by INDAM-GNAMPA project 2025: \say{Ottimizzazione spettrale, geometrica e funzionale} and by the PRIN project 2022R537CS \say{NO$^3$--Nodal Optimization, NOnlinear elliptic equations, NOnlocal geometric problems, with a focus on regularity} - CUP J53D23003850006.

\noindent{\bf Availability of data and materials:}

no datasets were generated or analyzed during the current study.

\bibliographystyle{abbrv}
\bibliography{bibfrac2}
\end{document}